\documentclass[11pt]{amsart}
\usepackage{stmaryrd}
\usepackage{tikz-cd}
\usepackage[utf8]{inputenc}
\usepackage{fullpage, mathrsfs, mathtools, amssymb, tikz,bm, amscd, scrextend,  verbatim, enumerate, amsmath, eucal}
\usepackage[colorlinks=true, linkcolor=red!80!black, urlcolor=purple, citecolor=blue!70!black]{hyperref}
\usepackage[inline]{enumitem}
\usepackage{microtype}
\usepackage{verbatim}
\usepackage{caption}

\usetikzlibrary{matrix, arrows}

\input xy
\xyoption{all} 

\usepackage{esvect}

\newtheorem{theorem}{Theorem}[section]
\newtheorem{lemma}[theorem]{Lemma}
\newtheorem{prop}[theorem]{Proposition}
\newtheorem{cor}[theorem]{Corollary}

\theoremstyle{definition}
\newtheorem{definition}[theorem]{Definition}
\newtheorem{example}[theorem]{Example}
\newtheorem{remark}[theorem]{Remark}
\newtheorem*{list1}{Weierstrass Limits}
\newtheorem{assumption}[theorem]{Assumption}

\theoremstyle{remark}

\newcounter{step}

\usepackage{faktor}\usepackage{amsmath}\usepackage{amssymb}

\makeatletter
\DeclareRobustCommand*{\mfaktor}[3][]
{
	{ \mathpalette{\mfaktor@impl@}{{#1}{#2}{#3}} }
}
\newcommand*{\mfaktor@impl@}[2]{\mfaktor@impl#1#2}
\newcommand*{\mfaktor@impl}[4]{
	\settoheight{\faktor@zaehlerhoehe}{\ensuremath{#1#2{#3}}}%
	\settoheight{\faktor@nennerhoehe}{\ensuremath{#1#2{#4}}}%
	\raisebox{-0.5\faktor@zaehlerhoehe}{\ensuremath{#1#2{#3}}}%
	\mkern-4mu\diagdown\mkern-5mu%
	\raisebox{0.5\faktor@nennerhoehe}{\ensuremath{#1#2{#4}}}%
}

\newcommand{\G}{{\mathbb G}}

\newcommand{\LL}{{\mathbb L}}

\newcommand{\bP}{{\mathbb P}}

\newcommand{\bQ}{{\mathbb Q}}

\newcommand{\calK}{\mathcal K}
\newcommand{\calO}{\mathcal O}
\newcommand{\calM}{\mathcal M}

\newcommand{\calA}{\mathcal A}
\newcommand{\calB}{\mathcal B}
\newcommand{\calN}{\mathcal N}
\newcommand{\calW}{\mathcal W}
\newcommand{\calE}{\mathcal E}

\newcommand{\calC}{\mathcal C}
\newcommand{\calP}{\mathcal P}

\newcommand{\wi}{\mathrm{W}_{\mathrm{I}}}
\newcommand{\wii}{\mathrm{W}_{\mathrm{II}}}
\newcommand{\wiii}{\mathrm{W}_{\mathrm{III}}}

\newcommand{\SL}{\mathrm{SL}}
\newcommand{\PGL}{\mathrm{PGL}}
\newcommand{\Sym}{\mathrm{Sym}}
\newcommand{\Pic}{\mathrm{Pic}}

\newcommand{\mb}[1]{\mathbb{#1}}

\newcommand{\oMG}{\overline{\mathrm{W}}^G}
\newcommand{\MG}{{\mathrm{W}}^G}
\newcommand{\oMGslc}{\overline{\mathrm{W}}^G_{slc}}
\newcommand{\oMGL}{\overline{\mathrm{W}}^G_L}
\newcommand{\oMGslco}{\overline{\mathrm{W}}^G_{slc,o}}
\newcommand{\oMGLo}{\overline{\mathrm{W}}^G_{L,o}}
\newcommand{\sbb}{\overline{\mathrm{W}}^{*}}

\newcommand{\ka}{\overline{\calW}(\calA)}
\newcommand{\ska}{\overline{\calW}_\sigma(a)}
\newcommand{\ktwelve}{\overline{\calW}_\sigma({\frac{1}{12}+\epsilon})}

\newcommand{\ktwelvee}{\overline{\mathcal{W}}_\sigma({\frac{1}{12}-\epsilon})}

\newcommand{\ke}{\overline{\mathcal{K}}_{\epsilon}}
\newcommand{\me}{\overline{\mathrm{K}}_\epsilon}
\newcommand{\fe}{\overline{\mathcal{F}}_\epsilon}
\newcommand{\kepsilon}{\overline{\mathcal{W}}_\sigma(\epsilon)}
\newcommand{\kepsiloncoarse}{{\overline{\mathrm{W}}}_\sigma(\epsilon)}

\newcommand{\bru}{\overline{\mathscr{B}}}
\newcommand{\km}{\fe}

\newcommand{\oMGd}{\widetilde{\mathrm{W}}^G}

\usepackage[font=scriptsize,labelfont=bf]{caption}
\renewcommand{\footnotesize}{\scriptsize}

\title{Compact Moduli of elliptic K3 surfaces}
\author{Kenneth Ascher \& Dori Bejleri}

\begin{document}
	
	\maketitle
	
	\begin{abstract}
		We construct various modular compactifications of the space of elliptic K3 surfaces using tools from the minimal model program, and explicitly describe the surfaces parametrized by their boundaries. The coarse spaces of our constructed compactifications admit morphisms to the Satake-Baily-Borel compactification and the GIT compactification of Miranda.
	\end{abstract}
	
	\section{Introduction}

	Ever since the compactification of the moduli space of smooth curves by Deligne-Mumford was accomplished, the search for analogous compactifications in higher dimensions became an actively studied problem in algebraic geometry. While moduli in higher dimensions is highly intricate, the pioneering work of Koll\'ar-Shepherd-Barron \cite{ksb} and Alexeev \cite{boundedness}, (see also \cite{hx1, hmx, kp, kolmodbook}, etc.) has established much of the underlying framework for modular compactifications in the (log) general type case via \emph{KSBA stable pairs}, where semi-log canonical singularities serve as the generalization of nodal curves (see the survey \cite{kollarmumford}).

	One of the most sought after compactifications is for the space of K3 surfaces. K3 surfaces do not immediately fit into the above framework as they are not of general type, but rather Calabi-Yau varieties. On the other hand, like for abelian varieties, since the space of (polarized) K3 surfaces is a locally symmetric variety, it has several natural compactifications, e.g. the Satake-Baily-Borel (SBB), toroidal, and semi-toric compactifications of Looijenga. Unlike the KSBA approach, these compactifications do not necessarily carry a universal family or modular meaning over the boundary. 
	
	As such, one of the central questions in moduli theory is to give the aforementioned naturally arising compactifications a stronger geometric meaning by connecting them with a KSBA compactification. With this in mind, the goal of this paper is to construct modular compactifications for elliptic K3 surfaces -- compactifications where the degenerate objects are K3 surfaces with controlled singularities -- and understand how they compare to the Satake-Baily-Borel compactification.

	By the Torelli theorem, the moduli space of polarized K3 surfaces is a 19 dimensional locally symmetric variety. Similarly, it is well known that the moduli space of elliptic K3 surfaces with a section, which we denote by $W$, is an 18 dimensional locally symmetric variety, corresponding to $U$-polarized K3 surfaces  (see \cite{dolgachev,nikulin}). Recall that a generic elliptic K3 surface $f : X \to \bP^1$ with section $S$ has 24 $\mathrm{I}_1$ singular fibers. Let $F_{\calA} = \sum a_i F_i$ denote the sum of these 24 fibers weighted by $a_i \in \bQ \cap [0,1]^{24}$.  We consider the closure of the locus of pairs $(f: X \to C, S + F_\calA)$ inside the KSBA moduli space. For the moment, we assume all $a_i = a$, so that we can quotient by $S_{24}$. Denote the closure of the resulting locus by $\ska$, and let $0 < \epsilon \ll 1$.
	
	\begin{theorem}(see Theorem \ref{thm:kepsilon}, Theorem \ref{thm:git2}, Theorem \ref{thm:appear}, and Figure \ref{eq:diagram})\label{thm:intro2} The proper Deligne-Mumford stacks $\ska$ for $a \in \bQ \cap [0,1]$ give modular compactifications of $W$. There is an explicit classification of the broken elliptic K3 surfaces parametrized by $\kepsilon$, and an explicit morphism from the coarse space $\kepsiloncoarse \to \sbb$ to the SBB compactification of $W$. Furthermore, the surfaces paramaterized by $\kepsilon$ satisfy $\mathrm{H}^1(X, \calO_X) = 0$ and $\omega_X \cong \calO_X$. 
	\end{theorem}
	
	Theorem \ref{thm:intro2} shows that the boundary of $\kepsilon$ parametrizes K3 surfaces with  slc singularities. Although $\kepsilon$ compactifies a moduli spaces of pairs, it gives a natural compactification of the space of elliptic K3s as the choice of fibers is intrinsic. Indeed, without a boundary divisor, the moduli space is a non-separated Artin stack. In Section \ref{sec:explicitboundary}, we present an alternative explicit description of the surfaces parametrized on the boundary more akin to Kulikov models. In particular, we show that we can decompose the boundary of $\ska$ into combinatorially described parameter spaces.

	As mentioned above, viewing the moduli space of elliptic K3 surfaces as a locally symmetric variety, one naturally obtains the SBB compactification $W^*$. While a priori the SBB compactification does not have a modular meaning, it turns out that in the case of elliptic K3 surfaces, this compactification can be identified with the GIT compactification of Weierstrass models of Miranda $\oMG$ (see Section \ref{sec:gitbackground} and \cite[Theorem 7.9]{oo}), which provides some geometric meaning. In particular, in the theorem above, as well as the remainder of this section, all of our spaces admit morphisms to $\oMG$.
	

	One benefit of the SBB compactification is that all of the parametrized surfaces are irreducible. The next theorem discusses a modular compactification, coming from the KSBA approach, where the boundary parametrizes irreducible surfaces. 
	Indeed, consider pairs $(f: X \to \bP^1, S + \epsilon F)$ for $0 < \epsilon \ll 1$, i.e. only one singular fiber carries a non-zero weight, and this weight is infinitesimally small. We denote the closure of this locus by $\ke$.

	\begin{theorem}(see Theorem \ref{thm:boundary}, Theorem \ref{thm:git}, and Figure \ref{eq:diagram})\label{thm:intro3}
		The compact moduli space $\ke$ parametrizes irreducible semi-log canonical Weierstrass K3 surfaces satisfying $\mathrm{H}^1(X, \calO_X) = 0$ and $\omega_X \cong \calO_X$. Moreover, there is an explicit generically finite morphism from the coarse space $\me \to \sbb$.\end{theorem}
	
	In light of the above theorem, it is natural to ask how the compactifications $\kepsilon$ and $\ke$ are related. In previous work (see \cite{master}) we showed the existence of wall-crossing morphisms on moduli spaces of elliptic surfaces. In particular, our previous work implies that (up to a $24$-to-$1$ base change corresponding to choosing a singular fiber) the universal families of $\kepsilon$ and $\ke$ are related by an explicit series of flips and divisorial contractions as the weights of 23 of the marked fibers are reduced from $\epsilon$ to 0. This aspect is crucial to our work (see e.g. Section \ref{sec:onemarked}) -- these explicit morphisms allow us to understand how our compactifications are related to each other, and how they compare to others lacking a modular meaning.

	Finally, we introduce one more KSBA compactification. While in $\ke$ we mark one singular fiber with weight $\epsilon$, it is natural to ask what happens if we mark \emph{any} fiber, not necessarily singular, with weight $\epsilon$. We denote this compactification by $\km$. Before stating the final theorem of the introduction, we point the reader to Figure \ref{eq:diagram} for an overview of the spaces introduced in this paper. 
	
	\medskip

	\begin{minipage}{0.45\textwidth}
		
		\begin{equation}\label{eq:diagram}
			\begin{tikzcd}
				& \ka \ar[d] \ar[ld]  \\
				\ska \ar[rd] \ar[d,"\cong"] & \ke \ar[d] \ar[hookrightarrow]{r}& \km \ar[d] \\
				\bru^\nu  & \sbb \cong \oMG & \oMGd \ar[l] \end{tikzcd} \end{equation}\captionof{figure}{This diagram shows the various compactifications we introduce in this paper as well as how they are related (see also Remark \ref{rem:variousmoduli}).}
	\end{minipage}
	\begin{minipage}{0.45\textwidth}
		\footnotesize
		\begin{itemize}
			\item[$\bru^\nu$:] The normalization of  Brunyate's compactification with small weights on both section and singular fibers (see Section \ref{sec:brunyate}).
			\item[$\ka$:] KSBA compactification with $\calA$-weighted singular fibers.
			\item[$\ska$:] When $\calA = (a, \dots, a)$, we quotient by $\mathrm{S}_{24}$. 
			\item[$\ke$:] KSBA compactification with a single $\epsilon$-marked singular fiber (where $\epsilon \ll 1$).
			\item[$\km$:] KSBA compactification with \emph{any fiber} marked by $\epsilon$ (where $\epsilon \ll 1$).
			\item[$\sbb$:] SBB compactification of the period domain moduli space $W$.
			\item[$\oMG$:] Miranda's GIT compactification of Weierstrass models (see Section \ref{sec:gitbackground}).
			\item[$\oMGd$:] GIT compactification of Weierstrass models with a chosen fiber (see discussion after Theorem \ref{thm:mainintro}).
		\end{itemize}
	\end{minipage}
	
	\medskip
	\medskip

	\begin{theorem}(see Theorem \ref{thm:mainresult} and Figure \ref{eq:diagram})\label{thm:mainintro} There exists a smooth proper Deligne-Mumford stack $\km$ parametrizing semi-log canonical elliptic K3 surfaces with a single marked fiber. Its coarse space is isomorphic to an explicit GIT quotient $\oMGd$ of Weierstrass K3 surfaces and a chosen fiber. Furthermore, the surfaces parametrized by $\km$ satisfy $\mathrm{H}^1(X, \calO_X) = 0$ and $\omega_X \cong \calO_X$.
	\end{theorem}
	
	On the interior, $\km$ is a $\bP^1$ bundle over $W$. In this sense, $\km$ is similar in spirit to the KSBA compactification of Laza of degree two K3 surfaces \cite{laza}. The GIT problem of Miranda can be modified  to parametrize Weierstrass fibrations with a chosen fiber (see Section \ref{sec:git2}), denoted above by $\oMGd$. It turns out that $\oMGd$ is precisely the coarse moduli space of $\km$ -- in particular, the morphism $\km \to \oMGd$ realizes $\km$ as a smooth Deligne-Mumford stack.

	Our approach in this paper combines explicit use of the theory of twisted stable maps (see e.g. \cite{tsm}) with the minimal model program. The various compactifications are then related by an explicit series of wall-crossing morphisms.  In particular, we wish to emphasize that the power of our approach lies in understanding the compactifications for various coefficients and how they are related via wall crossing morphisms.  Often the spaces with very small coefficients are the smallest compactifications which are still modular, but having access to the spaces for all coefficients is fruitful in understanding the geometry of compactifications obtained via different methods.


	


	\subsection{Previous results}\label{sec:brunyate}
	Using  Kulikov models, Brunyate's thesis \cite{brunyate} constructs a stable pairs compactification of the space of elliptic K3 surfaces $\bru$ which parametrizes pairs $(X, \epsilon S + \delta F)$, where $\epsilon$ and $\delta$ are both small. In particular, Brunyate gives a classification of the surfaces appearing on the boundary, and conjectures that the normalization of $\bru$ is a toroidal compactification.  Recently, Alexeev-Brunyate-Engel \cite{abe} confirmed Brunyate's conjecture, and showed that this space is isomorphic to a particular toroidal compactification using the theory of integral affine geometry and continuing the program started in \cite{aet}. 
	
	One difference between our approach and the work of Brunyate, is in our descriptions of the compactifications at various weights and choice of markings. Instead of using Kulikov models, we describe the steps of MMP and the induced wall-crossing morphisms that relate the stable limits of elliptic K3 surfaces for different weights to highlight the underlying geometry of the various compactifications. Brunyate's space $\bru$ admits a morphism $\overline{\mathcal{W}}_\sigma(\epsilon) \to \bru$ which identifies $\overline{\mathcal{W}}_\sigma(\epsilon)$ with the normalization of $\bru$ (see Proposition \ref{prop:brunyate} and Remark \ref{rem:brunyate2}). In particular, the boundary components of $\bru$ and $\kepsilon$ are in bijection (see Remark \ref{rmk:brunyate}) and the moduli spaces parametrize essentially the same surfaces. Indeed there is a sequence of flips relating the universal family of $\bru$ and the universal family over $\overline{\mathcal{W}}_\sigma(\epsilon)$ which induces this morphism. 
	
	Finally, we note that in a slightly different direction, Inchiostro constructs a KSBA compactification of the space of Weierstrass fibrations (of not necessarily K3 surfaces) with both section and fibers marked by $0 < \epsilon, \delta \ll 1$ \cite{giovanni},

	\subsection{Other lattice polarizations}
	It is natural to consider fibrations with specified singular fibers. In this case, one obtains a moduli space which is a locally symmetric variety, corresponding to a $M$-lattice polarization, encoding the singular fiber type. Our methods work in that case as well. Here we quickly discuss an example of this point of view.

	
	
	\begin{example} Consider the lattice $M = U \oplus D_4^{\oplus 4}$. Then $M$-polarized K3 surfaces correspond to $4\mathrm{I}_0^*$ isotrivial elliptic K3 surfaces. Equivalently, these are Kummer K3 surfaces obtained from abelian surfaces of the form $E \times E'$ with the elliptic fibration induced by the projection $E \times E' \to E$. Marking the $4$ minimal Weierstrass cusps by a single weight $a$ gives us a moduli space whose coarse space is two copies of the $j$-line, one parametrizing the $j$-invariant of the fibration, and the other the $j$-invariant of the configuration of singular fibers. The stable pairs compactification has coarse space given by $\mathbb{P}^1 \times \mathbb{P}^1 = \overline{M}_{0,4} \times \overline{M}_{0,4}$. The universal family consists of $4\mathrm{N}_1$ isotrivial $j$-invariant $\infty$ fibrations over the locus $\{\infty\} \times \mathbb{P}^1$, a union $X \cup_{\mathrm{I_0}} X$ of two copies of the $2\mathrm{I}_0^*$ rational elliptic surface glued along a smooth fiber over the locus $\mb{P}^1 \times \{\infty\}$, and a union $X \cup_{\mathrm{N}_0} X$ of two copies of the $2\mathrm{N}_1$ isotrivial $j$-invariant $\infty$ fibration glued along an $\mathrm{N}_0$ fiber over the point $(\infty, \infty)$. 
	\end{example} 
	
	\subsection{Structure of the paper}
	In \textbf{Section} \ref{sec:intromoduliK3} we discuss the background on elliptic K3 surfaces and their moduli (as a period domain, the Satake-Baily-Borel compactification, and a Geometric Invariant Theory compactification). In \textbf{Section} \ref{sec:broken} we review the results from our previous works (\cite{calculations, tsm, master, tsm2}) on KSBA compactifications of moduli spaces of elliptic fibrations and the connection with twisted stable maps. In \textbf{Section} \ref{sec:weightedstable} we restrict to the case of elliptic K3 surfaces and collect the definitions of and preliminary observations on the compactifications we consider in this paper, including a discussion on isotrivial $j$-invariant $\infty$ fibrations of K3 type. 
	
	The main body of the paper begins with \textbf{Section} \ref{sec:firstwall} where we discuss the wall-crossings that occur for the compactification $\ska$ as the coefficient $a$ is lowered from $1$ down to $1/12 + \epsilon$ for $0 < \epsilon \ll 1$.  In \textbf{Section} \ref{sec:epsilon} we continue the wall-crossing analysis as $a$ is decreased down to $0 < \epsilon \ll 1$, and we prove \textbf{Theorem} \ref{thm:intro2}, which describes the surfaces appearing on the boundary of the moduli space $\kepsilon$. In \textbf{Section} \ref{sec:explicitboundary} we use Theorem \ref{thm:intro2} and twisted stable maps (Section \ref{sec:tsm}) to explicitly describe the boundary components of $\kepsilon$. Finally, in \textbf{Section} \ref{sec:onefiber} we describe the moduli spaces with one marked fiber $(\ke$ and $\km$) and prove \textbf{Theorem} \ref{thm:intro3} and \textbf{Theorem} \ref{thm:mainintro}; the latter theorem is proven by introducing a modified version of Miranda's GIT compactification (see Section \ref{sec:git2}).

	\subsection*{Acknowledgments} We thank Valery Alexeev, Izzet Coskun, Kristin DeVleming, Giovanni Inchiostro,  J\'anos Koll\'ar, Radu Laza, Yuchen Liu,  Siddharth Mathur, Yuji Odaka, and David Yang. We thank Adrian Brunyate for pointing out a gap in a previous version. We thank the referees for their very helpful comments and suggestions. Both authors supported by NSF Postdoctoral Fellowships.  Part of this paper was written while K.A. was in residence at Mathematical Sciences Research Institute in Berkeley, CA, during the Spring 2019, supported by the National Science Foundation under Grant No. 1440140. K.A. partially supported by NSF grant DMS-2140781 (formerly DMS-2001408). 
	
	\section{Elliptic K3 surfaces and their moduli} \label{sec:intromoduliK3}
	
	
	\subsection{Elliptic surfaces} We begin with the basic definitions surrounding elliptic surfaces following \cite{master} (see also \cite{mir3}).

	\begin{definition}
		An irreducible \textbf{elliptic surface with section} ($f: X \to C, S)$ is an irreducible surface $X$ together with a surjective proper flat morphism $f: X \to C$ to a smooth curve $C$ and a section $S$ such that:
		\begin{enumerate} 
			\item the generic fiber of $f$ is a stable elliptic curve, and
			\item the generic point of the section is contained in the smooth locus of $f$. 
		\end{enumerate}
		We call the pair $(f: X \to C, S)$ \textbf{standard} if all of $S$ is contained in the smooth locus of $f$.\end{definition}
	
	\begin{definition}\label{def:weierstrassmodel} A \textbf{Weierstrass fibration} is an elliptic surface obtained from a standard elliptic surface by contracting all fiber components not meeting the section. We call the output of this process a \textbf{Weierstrass model}. If starting with a smooth relatively minimal elliptic surface, we call the result a \textbf{minimal Weierstrass model}.  \end{definition}
	
	The geometry of an elliptic surface is largely influenced by the \emph{fundamental line bundle} $\mathscr L$.
	
	\begin{definition} The \textbf{fundamental line bundle} of a standard elliptic surface is $\mathscr L := (f_* \mathcal{N}_{S/X})^{-1}$, where $\mathcal{N}_{S/X}$ denotes the normal bundle of $S$ in $X$. For an arbitrary elliptic surface we define $\mathscr{L}$ as the line bundle associated to its minimal semi-resolution\footnote{the semi-normal version of resolution of singularities, see e.g. \cite[Section 1.13]{singmmp}}. \end{definition}
	
	For $X$ a standard elliptic surface, the line bundle $\mathscr{L}$ is invariant under taking a semi-resolution or Weierstrass model, is independent of choice of section $S$, has non-negative degree, and determines the canonical bundle of $X$ if $X$ is either relatively minimal or Weierstrass (see \cite[III.1.1]{mir3}). 
	
	\subsection{Singular fibers} If $(f: X \to C, S)$ is a smooth relatively minimal elliptic surface, then $f$ has finitely many singular fibers which are each unions of rational curves with possibly non-reduced components whose dual graphs are ADE Dynkin diagrams. The singular fibers were classified by Kodaira-Ner\'on (see \cite[Section V.7]{complexsurfaces}). 
	
	An elliptic surface in Weierstrass form can be described locally by an equation of the form $y^2 = x^3 + Ax + B$ where $A$ and $B$ are functions of the base curve. Furthermore, the possible singular fiber types  can be characterized in terms of vanishing orders of $A$ and $B$ by Tate's algorithm (see \cite[Table 1]{sselliptic}). Moreover, if the smooth relatively minimal model $(f : X \to C, S)$ has a singular fiber with a given Dynkin diagram, the minimal Weierstrass  model will have an ADE singularity of the same type.

	\subsection{Elliptic K3 surfaces}
	
	By the canonical bundle formula and the observation that $\deg \mathscr{L} = 0$ if and only if the surface is a product, a smooth elliptic surface with section $(f : X \to C, S)$ is a K3 surface if and only if $C \cong \mb{P}^1$ and $\deg(\mathscr{L}) = 2$ (see \cite[III.4.6]{mir3}). 
	
	\begin{definition} A standard (possibly singular) elliptic surface is of \textbf{K3 type} if $C \cong \bP^1$ and $\deg(\mathscr{L}) = 2$. \end{definition}
	
	For an elliptic surface of K3 type, the Weierstrass model is given by $y^2 = x^3 + Ax + B$, where $A$ and $B$ are sections of $\calO(8)$ and $\calO(12)$ respectively, and the \emph{discriminant} $\mathscr{D} = 4A^3 + 27B^2$ is a section of $\mathscr{L}^{\otimes 12} \cong \calO(24)$.

	\begin{remark}\label{rmk:singfibers} The number of singular fibers of a Weierstrass elliptic K3 counted with multiplicity is 24, and a generic elliptic K3 has exactly 24 nodal $(\mathrm{I}_1)$ singular fibers. \end{remark}

	We now discuss lattice polarized K3 surfaces and their moduli (see \cite{HT, friedman, friedmanbook}).

	\subsection{Moduli of lattice polarized K3 surfaces}
	An elliptic K3 with section $(f: X \to \bP^1, S)$ are characterized by the fact that $\rm{NS}(X)$ contains a lattice $U$ which is spanned by the classes of the fiber $f$ and section $S$. The moduli of K3 surfaces with specified $\rm{NS}(X)$ were studied by Dolgachev \cite{dolgachev} (see also \cite{nikulin}). By the Torelli theorem for polarized K3 surfaces, the moduli space of minimal Weierstrass elliptic K3 surfaces with at worst ADE singularities is an 18-dimensional locally symmetric variety $W = \Gamma\backslash D$ associated to the lattice $U^\perp_{\textrm{K}3} \cong U^2 \oplus E^2_8$.

	\subsection{The Satake-Baily-Borel compactification}\label{sec:bb} One can use the techniques of Baily-Borel \cite{bb} to obtain a compactification $\sbb$ by adding some curves and points. We briefly review this compactification following  \cite[Section 3.1]{lazacpd}. The boundary components of $\sbb$ are determined
	by rational maximal parabolic subgroups of the identity component of the orthogonal group $O(2,18)$ of the lattice $U_{\textrm{K}3}^\perp$. Every boundary component of $\sbb$ has the structure of a locally symmetric variety of lower dimension. Furthermore, we recall the following properties:
	\begin{enumerate}
		\item The compactification is canonical.
		\item The boundary components have high codimension (as they are points and curves).
		\item It is \emph{minimal}: if $S$ is a smooth variety with $\overline{S}$ a smooth simple normal crossing compactification, then any locally liftable map $S \to W$ extends to a regular map $\overline{S} \to \sbb$.
	\end{enumerate}

	\begin{theorem}(see \cite[Section 2.3]{HT} and \cite{scattone}) The boundary of $\sbb$ is a union of zero and one dimensional strata. The 0-dim strata correspond to K3s of Type III, and the 1-dim strata to degenerate K3’s of Type II. Moreover, the 1-dim strata are all rational curves, each parametrizing the $j$-invariant of the elliptic double curves appearing in the corresponding Type II degenerate K3.  \end{theorem}

	\subsection{Geometric invariant theory}\label{sec:gitbackground}
	Miranda \cite{mir} used geometric invariant theory (GIT) to construct a compactification of the moduli space of Weierstrass fibrations, and completed an explicit classification in the case of rational elliptic surfaces. More recently, Odaka-Oshima \cite{oo} explicitly calculated Miranda's compactification for the case of elliptic K3 surfaces. Moreover, they showed that the GIT compactification of Miranda $\oMG$ is isomorphic to $W^*$, the SBB compactification. In particular, using this identification, one is able to give a geometric meaning to $W^*$ by relating the boundary of $W^*$ with the GIT polystable orbits in $\oMG$. We review these results now.

	Let $\Gamma_n = \Gamma(\bP^1, \calO_{\bP^1}(n))$. The surface $X$ has a \emph{Weierstrass equation}, and as such $X$ can be realized as a divisor in a $\bP^2$-bundle over the base curve. For the Weierstrass model of an elliptic K3 surface, we think of $X$ as being the closed subscheme of $\bP(\calO_{\bP^1}(4) \oplus \calO_{\bP^1}(6) \oplus \calO_{\bP^1})$ defined by the equation $y^2z = x^3 + Axz^2 + Bz^3,$ where $A \in \Gamma_8$, $B \in \Gamma_{12}$,  and 
	\begin{enumerate}
		\item $4A(q)^3 + 27B(q)^2 = 0$ precisely at the (finitely many) singular fibers $X_q$, 
		\item and for each $q \in \bP^1$ we have $v_q(A) \leq 3$ or $v_q(B) \leq 5$. 
	\end{enumerate} 
	
	We note that any Weierstrass elliptic K3 surface (with section) and ADE singularities satisfies the above conditions, and conversely, the surface defined as above is a Weierstrass elliptic K3 surface with section and ADE singularities (see \cite[Theorem 7.1]{oo}). 
	
	We denote by $V_{24} = \Gamma_8 \oplus \Gamma_{12}$ and the GIT moduli space for Weierstrass elliptic K3 surfaces by $\oMG = V_{24}^{ss} \sslash SL_2$. By the above discussion the open locus $\MG \subset \oMG$ parametrizes the ADE Weierstrass elliptic K3 surfaces. The following theorem describes the boundary $\oMG \setminus \MG$.

	\begin{theorem}\cite[Proposition 7.4]{oo}\label{thm:gitsss} The boundary $\oMG \setminus \MG$ is as follows, namely there is a
		\begin{enumerate}
			\item 1-dimensional component $\oMGslc$ parametrizing isotrivial $j$-invariant $\infty$ slc surfaces.
			\item  1-dimensional component $\oMGL$ whose open locus $\oMGLo$ parametrizes normal surfaces with 2 type L type cusps. 
		\end{enumerate}
		Furthermore, the intersection of the two components is the infinity point of both $\bP^1$s parametrizing the unique $j$-invariant $\infty$ slc surface with two L type cusps. This point is polystable, and the strictly semistable locus is $\oMGL$, i.e $\oMGslc$ is part of the GIT-stable locus of $\oMG$.
	\end{theorem}
	
	A natural question is how the GIT compactification $\oMG$ compares to the SBB compactification $\sbb$. This is the content of \cite[Theorem 7.9]{oo}, where we denote $\oMGslco := \oMGslc \setminus \oMGL$.
	
	\begin{theorem}\cite[Theorem 7.9]{oo}
		The period map $\MG \to W$  extends to an isomorphism $\oMG \cong \sbb$. Moreover, the above isomorphism identifies $\oMGslco \cup \oMGLo$ with the 1-dimensional cusps, and identifies $\oMGslc \cap \oMGL$ with the 0-dimensional cusp.
	\end{theorem}

	\section{Moduli of $\calA$-broken elliptic surfaces and wall-crossing}\label{sec:broken}
	
	In this section we review and supplement the results from our previous work on compactifications of the moduli spaces of elliptic surfaces via KSBA stable pairs. 
	
	\begin{definition} A \textbf{KSBA stable pair} $(X,D)$ is a pair consisting of a variety $X$ and a Weil divisor $D$ such that 
		\begin{enumerate} 
			\item $(X,D)$ has semi-log canonical (slc) singularities, and
			\item $K_X + D$ is an ample $\mathbb{Q}$-Cartier divisor. 
		\end{enumerate}
	\end{definition}
	
	Stable pairs are the natural higher dimensional generalization of stable curves and their moduli space compactifies the moduli space of log canonical models of pairs of log general type. 
	
	In \cite{master}, we defined KSBA compactifications $\calE_\calA$ of the moduli space of log canonical (lc) models $(f: X \to C, S + F_\calA)$ of $\calA$-weighted Weierstrass elliptic surface pairs. For each admissible weight vector $\calA$, we obtain a compactification $\calE_\calA$, which is representable by a proper Deligne-Mumford stack of finite type \cite[Theorem 1.1 \& 1.2]{master}. These spaces parameterize slc pairs $(f: X \to C, S + F_{\calA})$, where $(f: X \to C, S)$ is an slc elliptic surface with section, and $F_{\calA} = \sum a_i F_i$ is a weighted sum of marked fibers with $\calA = (a_1, \dots, a_n)$ and $0 < a_i \leq 1$, and $(X , S + F_\calA)$ is a stable pair. 
	
	Before stating the main result Theorem \ref{thm:masterthm}, we must first discuss the different (singular) fiber types that appear in semi-log canonical models of elliptic fibrations as studied in \cite{calculations}. 
	
	\begin{definition}\label{def:fibertypes} Let $(g: Y \to C, S' + aF')$ be a Weierstrass elliptic surface pair over the spectrum of a DVR and let $(f: X \to C, S + F_a)$ be its relative log canonical model. We say that $X$ has a(n):  
		\begin{enumerate} 
			\item \textbf{twisted fiber} if the special fiber $f^*(s)$ is irreducible and $(X,S + E)$ has (semi-)log canonical singularities where $E = f^*(s)^{red}$;   
			\item \textbf{intermediate fiber} if $f^*(s)$ is a nodal union of an arithmetic genus zero component $A$, and a possibly non-reduced arithmetic genus one component supported on a curve $E$ such that the section meets $A$ along the smooth locus of $f^*(s)$ and the pair $(X, S + A + E)$ has (semi-)log canonical singularities. 
		\end{enumerate} 
	\end{definition}
	
	Given an elliptic surface $f: X \to C$ over the spectrum of a DVR such that $X$ has an \emph{intermediate fiber}, we obtain the \emph{Weierstrass model}  of $X$ by contracting the component $E$, and we obtain the \emph{twisted model} by contracting the component $A$. As such, the intermediate fiber can be seen to interpolate between the Weierstrass and twisted models.

	One can consider a Weierstrass elliptic surface $(g: Y \to C, S' + aF')$ over the spectrum of a DVR, where $F'$ is either a Kodaira singular fiber type, or $g$ is isotrivial with constant $j$-invariant $\infty$ with $F'$ being an $\mathrm{N}_k$ fiber type. Then the relative log canonical model $(f: X\to C, S + F_a)$ depends on the value of $a$. When $a = 1$, the fiber is in \emph{twisted} form, when $a = 0$ the fiber is in \emph{Weierstrass form}, and for some $0 < a_0 < 1$, the fiber enters \emph{intermediate} form. The values $a_0$ were calculated for all fiber types in \cite[Theorem 3.10]{master}.
	
	\begin{equation}\label{eq:a0}
		a_0 = \left\{ \begin{array}{lr}
			
			5/6 & \mathrm{II} \\ 3/4 & \mathrm{III} \\ 2/3 & \mathrm{IV} \\  1/2 & \mathrm{N}_1 \end{array}\right. \\
		\  a_0 = \left\{ \begin{array}{lr} 1/6 & \mathrm{II}^* \\ 
			1/4  & \mathrm{III}^* \\
			1/3 & \mathrm{IV}^* \\
			1/2 & \mathrm{I}_n^* \end{array}\right. 
	\end{equation}

	We now state the definition of pseudoelliptic surfaces which appear as components of surfaces in our moduli spaces, a phenomenon first observed by La Nave \cite{ln}.
	
	\begin{definition}\label{def:pseudo} A \textbf{pseudoelliptic pair} is a surface pair $(Z, F)$ obtained by contracting the section of an irreducible elliptic surface pair $(f: X \to C, S + F')$. We call $F$ the \textbf{marked pseudofibers} of $Z$. We call $(f: X \to C, S)$ the associated elliptic surface to $(Z, F)$. \end{definition}
	
	The MMP will contract the section of an elliptic surface if it has non-positive intersection with the log canonical divisor of the surface. There are two types of pseudoelliptic surfaces which appear, and we refer the reader to \cite[Definition 4.6, 4.7]{master} for the precise definitions.

	\begin{definition}\label{def:pseudotypeII} A pseudoelliptic surface of \textbf{Type II}  is formed by the log
		canonical contraction of a section of an elliptic component attached along \emph{twisted} or \emph{stable} fibers. \end{definition}
	
	\begin{definition}\label{def:pseudotypeI} A pseudoelliptic surface of \textbf{Type I} appear in \emph{pseudoelliptic trees} attached by gluing an
		irreducible pseudofiber $G_0$ on the root component to an arithmetic genus one component $E$ of an
		intermediate (pseudo)fiber of an elliptic or pseudoelliptic component. \end{definition}

	Figure \ref{fig:pseudodef} has a tree of pseudoelliptic surfaces of Type I circled on the right, with a pseudoelliptic of Type II circled on the left.   
	
	\begin{theorem}\cite[Theorem 1.6]{master}\label{thm:masterthm} The boundary of the proper moduli space $\calE_{v,\calA}$ parametrizes
		$\calA$-broken stable elliptic surfaces, which are pairs $(f: X \to C, S + F_{\calA})$
		consisting of a stable pair $(X, S + F_\calA)$ with a map to a nodal curve $C$ such that:
		\begin{itemize}
			\item $X$ is an slc union of elliptic surfaces with section S and marked fibers, as well as \item chains of pseudoelliptic surfaces of type I and II (Definition \ref{def:pseudo}) contracted
			by $f$ with marked pseudofibers. \end{itemize} \end{theorem}

	\begin{figure}[!h]
		\includegraphics[scale=.65]{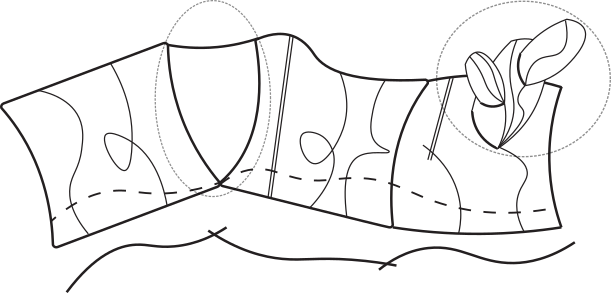}
		\caption{An $\calA$-broken elliptic surface. Two types of pseudoelliptic surfaces (see Definitions \ref{def:pseudotypeII} and \ref{def:pseudotypeI}) circled. Left: Type II and Right: Type I.}\label{fig:pseudodef}
	\end{figure}
	
	Contracting the section of a component to form a pseudoelliptic component corresponds to stabilizing the base curve
	as an $\calA$-stable curve in the sense of Hassett (see \cite[Corollaries 6.7 \& 6.8]{calculations}). In particular we have the following.
	
	\begin{theorem}\cite[Theorem 1.4]{master}\label{thm:forgetful} There are forgetful morphisms $\calE_{v, \calA} \to \overline{\calM}_{g, \calA}$. \end{theorem}

	\begin{remark}\label{rmk:contraction} For an irreducible component with base curve $\bP^1$ and $\deg \mathscr{L} > 0$,  contracting the section of an elliptic component may \emph{not} be the final step in the MMP -- we may need to contract the entire pseudoelliptic component to a curve or a point (see \cite[Proposition 7.4]{calculations}). \end{remark}
	
	\subsubsection{Wall and chamber structure} We are now ready to discuss how the moduli spaces $\calE_\calA$ change as we vary $\calA$. There are three types of walls in our wall and chamber decomposition.
	
	\begin{definition}\label{def:walltypes} \leavevmode \begin{enumerate}
			\item[(I)] A wall of \textbf{Type $\mathrm{W}_{\mathrm{I}}$} is a wall arising from the log canonical transformations, i.e. the walls where the fibers of the relative log canonical model
			transition  between fiber types.
			\item[(II)] A wall of \textbf{Type $\mathrm{W}_{\mathrm{II}}$} is a wall at which the morphism induced by the log canonical contracts the section of some components.
			\item[(III)] A wall of \textbf{Type $\mathrm{W}_{\mathrm{III}}$} is a wall where the morphism induced by the log canonical contracts an entire rational pseudoelliptic component (see Remark \ref{rmk:contraction}). \end{enumerate} \end{definition}

	\begin{remark}\label{rmk:hassettwall2}\leavevmode
		\begin{enumerate}
			\item The walls of Type $\mathrm{W}_{\mathrm{II}}$ are precisely the walls of Hassett's wall and chamber decomposition \cite{has} (see discussion preceding Theorem \ref{thm:forgetful}). 
			\item There are finitely many walls (see \cite[Theorem 6.3]{master}).
		\end{enumerate}
	\end{remark}

	\begin{theorem}\cite[Theorem 1.5]{master}\label{thm:main} Let $\calA, \calB \in \bQ^r$ be weight vectors with $0 < \calA \leq \calB \leq 1$. Then
		\begin{enumerate}
			\item If $\calA$ and $\calB$ are in the same chamber, then the moduli spaces and universal families are isomorphic.
			\item If $\calA \le \calB$ then there are reduction morphisms $\calE_{v,\calB} \to \calE_{v,\calA}$ on moduli spaces which are compatible with the reduction morphisms on the Hassett spaces:
			
			\item The universal families are related by a sequence of explicit divisorial contractions and flips 
			
			\noindent More precisely, across $\wi$ and $\wiii$ walls there is a divisorial contraction of the universal family and across a $\wii$ wall the universal family undergoes a log flip. 
	\end{enumerate}\end{theorem}
	
	\begin{remark}\label{rmk:flipping} For more on Theorem \ref{thm:main} (3), we refer the reader to \cite[Section 8]{master}. La Nave (see \cite[Section 4.3, Theorem 7.1.2]{ln}) noticed that the contraction of the section of a component is a log flipping contraction inside the total space of a one parameter degeneration.  In particular, the Type I pseudoelliptic surfaces are thus attached along the reduced component of an intermediate (pseudo)fiber (see \cite[Figure 13]{master}). \end{remark}
	
	\subsection{Strictly (semi-)log canonical Weierstrass models} 
	
	In order to understand the stable pair degenerations of log canonical models of Weierstrass elliptic surfaces, we need to understand strictly log canonical and semi-log canonical Weierstrass fibrations. We collect some results in this direction here, beginning with the definition of a \emph{type $\mathrm{L}$} singular fiber. 
	
	\begin{definition}(see \cite[Section 3.3]{ln})\label{def:l} Let $f: X \to C$ be a Weierstrass fibration with smooth generic fiber and Weierstrass data $(A,B)$. If $12 = \textrm{min}(3v_q(A), 2v_q(B))$ where $v_q$ denotes the order of vanishing at a point $q \in \bP^1$ we say that $f$ has a \textbf{type L fiber} at $q$. \end{definition}

	\begin{lemma}\label{lem:lcusp} If $F$ is a type L cusp of $X$ then $X$ has strictly log canonical singularities in a neighborhood of $F$ and the log canonical threshold $\mathrm{lct}(X,0,F)  = 0$. \end{lemma}
	
	\begin{proof} After performing a weighted blowup $\mu: Y \to X$ at the cuspidal point of $F$, we get an exceptional divisor $E$ a possibly nodal elliptic curve and strict transform $A:= \mu_*^{-1}(F)$ a rational curve meeting $E$ transversely. Writing $\mu^*K_X = K_Y + aE$, it follows from the projection formula that $K_Y.E + aE^2 = 0$. On the other hand, $K_Y.E + E^2 = K_E = 0$ by the adjunction formula and $E^2 \neq 0$ since it is exceptional. Therefore $a = 1$ so $X$ has a strictly log canonical singularity at the cuspidal point of $F$ and the discrepancy of $(X,\epsilon F)$ for any $\epsilon > 0$ will be strictly greater than $1$.
	\end{proof} 
	
	\begin{remark}\label{rmk:lcusp} The type L cusp decreases the self intersection $S^2$ by 1, and thus increases $\deg \mathscr{L}$ by 1 (see \cite[Remark 5.3.8]{ln}). \end{remark}
	
	We now discuss some facts on non-normal Weierstrass fibrations with generic fiber a nodal elliptic curve. These appear as semi-log canonical degenerations of normal elliptic surfaces and as isotrival $j$-invariant $\infty$ components of broken elliptic surfaces.

	We first recall the definition of the fiber types $\mathrm{N}_k$ that these fibrations possess (\cite[Section 5]{calculations} and \cite[Lemma 3.2.2]{ln}). 
	
	\begin{definition}\label{def:nk} The fibers $\mathrm{N}_k$ are the fiber types with Weierstrass equation $y^2 = x^2(x-t^k)$. \end{definition}
	
	\begin{lemma}\cite[Lemma 3.2.2]{ln} Fibers of type $\mathrm{N}_k$ are slc if and only if $k \in \{0, 1, 2 \}.$ \end{lemma}
	
	\begin{remark} \leavevmode
		\begin{enumerate}
			\item The general fiber of an isotrivial $j$-invariant $\infty$ fibration is type $\mathrm{N}_0$.
			\item $\mathrm{N}_2$ is the $j$-invariant $\infty$ version of the L cusp (see Remark \ref{rmk:n2}).
		\end{enumerate}
	\end{remark} 
	
	\begin{remark}\label{rmk:n2} The $\mathrm{N}_2$ fiber behaves analogously to the type L fiber. Indeed by the proof of \cite[Lemma 5.1]{calculations}, on the normalization $(X^\nu, D)$ of a surface $X$ with an $\mathrm{N}_2$ fiber, the double locus $D$ consists of a nodal curve with node lying over the cuspidal point of the $\mathrm{N}_2$ fiber, and $X^\nu$ is smooth in a neighborhood of this point. In particular, $(X^\nu, D)$ has log canonical singularities in a neighborhood of the nodal point of $D$ and $\mathrm{lct}(X^\nu, D, A) = 0$ for any curve $A$ passing through this point. Therefore by definition of semi-log canonical, $X$ has strictly semi-log canonical singularities in a neighborhood of the $\mathrm{N}_2$ fiber $F$ and $\mathrm{slct}(X,0,F) = 0$. \end{remark}

	The local equation given above for a type $\mathrm{N}_k$ fiber is not a standard Weierstrass equation. One can check that the standard equation of an $\mathrm{N}_k$ fiber is given by
	\begin{equation}\label{eqn:Nk}
		y^2 = x^3 - \frac{1}{3}t^{2k}x + \frac{2}{27} t^{3k}. 
	\end{equation}

	\begin{prop}\label{prop:jinftyfibers} If $(f : X \to C, S)$ is an isotrivial $j$-invariant $\infty$ slc Weierstrass fibration with $a_k$ type $\mathrm{N}_k$ fibers, then 
		$
		-S^2 = \deg(\mathscr{L}) = \sum_k a_k \frac{k}{2}.$
	\end{prop} 
	
	\begin{proof} Let $A$ and $B$ the Weierstrass data of $(f : X \to C, S)$. If $q \in C$ lies under an $\mathrm{N}_k$ fiber, then $A$ vanishes to order $2k$ and $B$ to order $3k$ at $q$. Then $A,B$ have degree $\sum 2ka_k$ and $\sum 3ka_k$ respectively. The result follows since the degree of $A$ and $B$ are $4\deg \mathscr{L}$ and $6\deg \mathscr{L}$ respectively. \end{proof} 
	
	Note that for $k$ even, the $\mathrm{N}_k$ fiber has trivial monodromy and for $k$ odd it has $\mu_2$ monodromy. This determines the twisted models of these fibers. 
	
	\begin{cor}\label{cor:monodromyNk} Let $F$ be an $\mathrm{N}_k$ fiber. Then the twisted model of $F$ is an $\mathrm{N}_0$ (respectively twisted $\mathrm{N}_1$) if $k$ is even (respectively odd).\end{cor}
	
	\begin{proof} By the local analysis of \cite[Section 6.2]{tsm}, in the even case the twisted model must be stable since there is no base change required, and the odd case there is a $\mu_2$ base change so the twisted model is a nodal cubic curve modulo the $\mu_2$ action, i.e. a twisted $\mathrm{N}_2$. \end{proof}
	
	Thus given an $\mathrm{N}_{k}$ fiber, we can cut it out and glue in an $\mathrm{N}_{k + 2}$ fiber since the families are isomorphic to $\mathrm{N}_0$ (respectively $\mathrm{N}_1$) families over a punctured neighborhood. We can ask how this surgery affects $-S^2 = \deg \mathscr{L}$.
	
	\begin{cor}\label{cor:jinftyfibers} Let $(f : X \to C, S)$ be an isotrivial $j$-invariant $\infty$ Weierstrass fibration and let $(f : X' \to C, S')$ be the result of replacing an $\mathrm{N}_k$ fiber by an $\mathrm{N}_{k + 2}$ fiber. Then $-(S')^2 = -S^2 + 1$. 
	\end{cor}

	\subsection{Elliptic fibrations via twisted stable maps}\label{sec:tsm}
	
	In \cite{tsm} we used the theory of twisted stable maps, originally developed by Abramovich-Vistoli (see \cite{av, av2}) to understand limits of families of elliptic fibrations. The basic idea is that an elliptic surface $f : X \to C$ gives an a priori rational map $C \dashrightarrow \overline{\calM}_{1,1}$ which extends to a morphism $\calC \dashrightarrow \overline{\calM}_{1,1}$ from an orbifold curve $\calC$ with coarse moduli space $C$. Now we understand limits of a family of elliptic surfaces by computing limits of the corresponding family of such maps. The twisted stable limits serve the same purpose for elliptic fibrations that \emph{Kulikov models} serve for K3 surfaces, i.e. they form the starting point from which applying the MMP yields the stable limit. 
	
	\subsubsection{Twisted stable maps limits} 
	
	We now recall structure of the limiting surfaces obtained using the twisted stable maps construction. As we will be studying slc degenerations of surfaces, the surfaces themselves will degenerate into possibly reducible surfaces. The degenerate surfaces will carry a fibration over a nodal curve whose $j$-map is the limit of the $j$-map of the degenerating family. Furthermore, there is a \emph{balancing condition} on the stabilizers of the orbicurve $\calC$ over nodes which implies the action on the tangent spaces of the two branches at a node must be dual (see \cite[Definition 3.2.4]{av} and \cite{olsson}). Finally, the stabilizers of a twisted stable map are concentrated either over nodes or at marked gerbes contained in the smooth locus. In particular, the limit of a map from a smooth schematic curve $C$ can only have stabilizers over the nodes. 
	
	These observations motivate the following necessary conditions for a twisted surface to appear as a limit of a family of degenerating elliptic surfaces. We consider the case where the degenerating family of elliptic surfaces has $12d\mathrm{I}_1$ marked singular fibers where $d = \deg \mathscr{L}$ as this is the generic situation and the relevant one for the present paper. This corresponds to the moduli map $C \to \overline{\calM}_{1,1}$ extending to a morphism on all of $C$ such that the $j$-map $C \to \overline{M}_{1,1} \cong \mathbb{P}^1$ has degree $12d$, and is unramified over $\infty$. 
	
	\begin{prop}\label{prop:tsm}\leavevmode Suppose $(f:X \to C, S+F)$ is a twisted elliptic surface \cite{tsm} over a rational curve which is the limit of a degenerating family of smooth elliptic surfaces with $12d\mathrm{I}_1$ and arbitrary marked fibers. Then the following hold. 
		\begin{enumerate}
			\item If $X$ is reducible, its irreducible components are either attached along nodal fibers, or in the following pairs of twisted fibers: $\mathrm{I}_a^*/\mathrm{I}_b^*/\mathrm{N}_1$, $\mathrm{II}/\mathrm{II}^*, \mathrm{III}/\mathrm{III}^*$ or $\mathrm{IV}/\mathrm{IV}^*$. 
			\item The total degree of the $j$-map $C \to \overline{M}_{1,1}$ is $12d$.
			\item Away from the singular locus of $C$, the fibers of $f$ are at worst nodal. In particular, every marked fiber in $F = \sum_{i = 1}^n F_i$ is an $I_a$ fiber for some $a \geq 0$. 
	\end{enumerate}\end{prop}
	
	The surfaces of Proposition \ref{prop:tsm} correspond to genus $0$ balanced twisted stable maps to $\overline{\calM}_{1,1}$ of degree $12d$ which are parametrized by the space $\calK_{0,n}(\overline{\calM}_{1,1}, 12d)(\underline{0})$. Here $\underline{0}$ is the tuple of $n$ zeroes denoting the fact that the marked points have trivial stabilizer. 
	
	\begin{theorem}\label{thm:tsm} \cite[Theorem 5.5]{tsm2} Each point $[(f : \calC \to \overline{\calM}_{1,1}, p_1, \ldots, p_{n})] \in \calK_{0,n}(\overline{\calM}_{1,1}, 12d)(\underline{0})$ admits a smoothing to a map from a non-singular $n$-pointed schematic rational curve. \end{theorem}
	
	\begin{cor}\label{cor:smoothing1} A twisted elliptic surface admits a smoothing to a generic $12d\mathrm{I}_1$ elliptic surface if and only if it satisfies the conditions of Proposition \ref{prop:tsm}. 
	\end{cor}
	
	\subsubsection{Relative twisted stable maps}
	
	One of the primary moduli spaces of interest from the perspective of stable pairs is the closure of the locus where the marked fibers are exactly the $12d\mathrm{I}_1$ fibers. These fibers lie above the preimages of $\infty \in \overline{\calM}_{1,1}$ under the $j$-invariant map $C \to \overline{\calM}_{1,1}$ and thus we are concerned with the closure $\calK_\infty \subset \calK_{0,24}(\overline{\calM}_{1,1},24)$ of the locus parametrizing maps from a smooth rational curve which are unramified over $\infty$ and such that all marked fibers map to $\infty$. Equivalently, this locus is the space of maps \emph{relative to the divisor $[\infty]$} with multiplicities $(1, \ldots, 1)$. The closure of such loci has been studied in the Gromov-Witten literature under the name of \emph{relative stable maps} (see e.g. \cite{vakil}, \cite{gathmann}, and \cite{cadman}). In \cite{tsm2}, we consider the question of determining the points of this locus for twisted stable maps to stacky curves. The conditions characterizing this locus \cite[Conditions $(\ast)$]{tsm2} can be phrased as follows in the context of elliptic fibrations. 
	
	\begin{prop}\label{prop:tsm2} Suppose $(f : X \to C, S + F)$ is a twisted elliptic surface over a rational curve which is the limit of a degenerating family of $12d\mathrm{I}_1$ elliptic surfaces with marked singular fibers. Then the following hold in addition to the conditions of Proposition \ref{prop:tsm}. 
		
		\begin{enumerate}
			\item $F$ consists of $12d$ nodal singular fibers.
			\item Every fiber with $j = \infty$ which is not on an isotrivial component is marked. 
			\item For each maximal connected tree $T$ of isotrivial $j = \infty$ components $X$, the number of marked fibers contained on $T$ is equal to the sum of the multiplicities of the twisted fibers of the non-isotrivial components along which $T$ is attached.
		\end{enumerate}
	\end{prop}
	
	\begin{remark} The last condition says e.g. that if an isotrivial $j$-invariant $\infty$ component is attached to an $\mathrm{I}_n$ fiber, there must be $n$ markings on that component, since an $\mathrm{I}_n$ fiber is produced when $n$ marked $\mathrm{I}_1$ fibers collide. \end{remark}

	\begin{theorem}\cite[Theorems 1.7 \& 1.8]{tsm2}\label{thm:tsm2} The conditions of Proposition \ref{prop:tsm2} characterize the boundary of $\calK_\infty$. In particular, any twisted surface satisfying these conditions is the limit of a family of smooth $12d\mathrm{I}_1$ elliptically fibered surface with marked singular fibers. 
	\end{theorem}

	\begin{remark} 
		
		After determining the shape of a twisted stable maps limits, we will use wall-crossing to compute the limits as one reduces weights.
	\end{remark}

	\section{Moduli of weighted stable elliptic K3 surfaces}\label{sec:weightedstable}
	
	In this section, we specialize the discussion of Section \ref{sec:broken} to the case of elliptic K3 surfaces and define the various compactifications of the stack $\calW$ of elliptic K3 surfaces, and its coarse space $W$, which we study in this paper. The goal is to obtain an \emph{explicit} description of the compactifications for various choices of weights $\calA$. In particular, we will explicitly describe the surfaces parametrized by the boundary of $\calE_\calA$ in this case as well as understand the wall-crossing morphisms.
	
	From now on we assume that $g(C) = 0$ and $\deg \mathscr{L} = 2$ so that $C \cong \mb{P}^1$ and $\mathscr{L} = \calO_{\mathbb{P}^1}(2)$ and $(f : X \to C, S)$ is an elliptic K3 surface with section.

	\begin{definition}\label{def:es} Let $\ka$ be the closure in $\mathcal{E}_{\calA}$ of the locus of pairs $(f : X \to C, S +F_\calA)$ where $X$ is an elliptic K3 surface and $\mathrm{Supp}(F_\calA)$ consists of 24 $\mathrm{I}_1$ singular fibers.\end{definition}
	
	\begin{definition} If $\calA = (a, \dots, a)$ is the constant weight\ vector, then $S_{24}$ acts on $\ka$ by permuting the marked fibers, and we denote the quotient by $\ska$. \end{definition}

	\begin{prop}\label{prop:latticeksba} $\ka$ and $\ska$ are proper Deligne-Mumford stacks. Moreover, the coarse space $\overline{W}_\sigma(a)$ of $\ska$ is a modular compactifications of $W$ for each $0 < a \le 1$.  \end{prop}
	
	\begin{proof}
		The fact that they are proper Deligne-Mumford stacks follows from \cite{master}. By construction, $\ska$ has an open set parametrizing elliptic K3s with $24\mathrm{I}_1$ fibers. Recall that $W$ parametrizes lattice polarized K3 surfaces, and such a lattice polarization is equivalent to the structure of an elliptic fibration with chosen section. The result follows by the observation that a generic elliptically fibered K3 surface has 24$\mathrm{I}_1$ fibers. \end{proof}

	Brunyate constructs a compactification $\bru$ of the space of elliptic K3 surface by studying degenerations of pairs $(X, \epsilon_1 S + F_{\calB})$ where $\calB = (\epsilon, \ldots, \epsilon)$, i.e. with small weight on both the section and the fibers (in particular, Brunyate requires $\epsilon_1 \ll \epsilon$), so that $\text{Supp}(F_\calB)$ is the closure of the rational curves on $X$ \cite{brunyate} (see also \cite[Section 7]{abe}). In fact there is a morphism $\bru^\nu \to \overline{\mathcal{W}}_\sigma(\epsilon)$, given by increasing the weight on the section to $1$. 
	
	\begin{prop}\label{prop:brunyate}There is a morphism $\bru^\nu \to \overline{\mathcal{W}}_\sigma(\epsilon)$ for $\epsilon \ll 1$. \end{prop}
	\begin{proof} Consider a 1-parameter degeneration of pairs $(X, \epsilon S + F_\calB)$ inside $\bru$. We may choose a generic choice of smooth fibers $G = \cup_{i \in I} G_i$ to mark so that the pair $(X, S + F_\calB + G)$ is stable, where the section has coefficient $1$. By the results of \cite{master}, there is a sequence of flips and contractions as one reduces the coefficients of $G$ from 1 to 0. The resulting stable limit in $\overline{\mathcal{W}}_\sigma(\epsilon)$ only depends on the point $(X_0, \epsilon S_0 + (F_\calB)_0)$ in $\bru$ and not on the family or choice of auxiliary markings. Therefore we obtain the desired morphism by \cite[Theorem 7.3]{gg}. \end{proof}

	\begin{remark}\label{rmk:brunyate} Comparing Theorem \ref{thm:kepsilon} with \cite[Theorem 9.1.4]{brunyate} (see also \cite[Section 7]{abe}), we see that there is a bijection between the boundary strata of $\bru$ and $\kepsilon = \mathcal{W}(\calB)/S_{24}$. For example, the third case in \cite[Theorem 9.1.4]{brunyate} maps to case (E) of Theorem \ref{thm:kepsilon} if there are no $\mathbb{F}_0$ components, and to either type (D) or (F) depending on the parity of the number of components if there are $\mathbb{F}_0$ components. 
	\end{remark} 
	
	\begin{cor}\label{cor:brunyate} The morphism from Proposition \ref{prop:brunyate} is an isomorphism. \end{cor} 
	
	\begin{proof} It is a proper birational set-theoretic bijection between normal spaces. \end{proof} 
	
	\begin{remark}\label{rem:brunyate2} It follows from Corollary \ref{cor:brunyate} that there is in fact a morphism $\kepsilon \to \bru$ which can be thought of as induced by decreasing weights on the section. \end{remark}

	\begin{definition} Let $\ke$ denote stable pairs compactification of the space parametrizing pairs with only one singular fiber marked with weight $0 < \epsilon \ll 1$, and let $\me$ be its coarse moduli space. 
	\end{definition}

	Next, we define the moduli space $\km$ which is like $\ke$, only we allow \emph{any} fiber to be marked.

	\begin{definition}\label{def:modulispace} Let $\km$ be the closure in $\mathcal{E}_{\calA}$ of the locus of pairs $(f: X \to C, S + \epsilon F)$ where $f$ has precisely 24 $\mathrm{I}_1$ fibers, $0 < \epsilon \ll 1$, and $F$ is any fiber. \end{definition}

	\begin{remark}\label{rem:variousmoduli} At this point we have introduced many compactifications (see Figure \ref{eq:diagram}):
		\begin{itemize}\setlength\itemsep{.5em}
			\item $\ka$: Stable pair compactification with $\calA$-weighted singular fibers.
			\item $\ska$: When $\calA = (a, \dots, a)$, we can quotient by $\mathrm{S}_{24}$. 
			\item $\ke$: Stable pairs compactification with a single $\epsilon$-marked singular fiber.
			\item $\km$: Stable pairs compactification with \emph{any fiber} marked by $\epsilon$.
			\item $\sbb$: SBB compactification of the period domain moduli space $W$.
			
		\end{itemize}
		
		We now give a brief overview of how they are related (again, see Figure \ref{eq:diagram}).
		
		\begin{enumerate}

			\item There are $24$ generically finite morphisms $\ka \to \ke$ of degree $23!$, corresponding to a forgetting all but one marked singular fiber.
			
			\item There is a degree $24$ generically finite rational map $\ke \dashrightarrow \kepsiloncoarse$ corresponding to choosing a singular fiber.

			\item We will see that there are morphisms $\kepsiloncoarse \to \sbb$ and $\me \to \sbb $ (see Theorems \ref{thm:git2} and \ref{thm:git} resp.).

			\item We will see in Section \ref{sec:git2} that the moduli space $\km$ is a smooth Deligne-Mumford stack whose coarse space is an (explicit) GIT quotient. Furthermore, there is a morphism $\km \to \sbb$ (see Theorem \ref{thm:mainresult}) which is generically a $\mathbb{P}^1$ bundle.
			
		\end{enumerate}
		
	\end{remark}
	
	We end this section with an important proposition. 
	
	\begin{prop}\label{obvs:dubois}For any surface $X$ parametrized by $\ka$ (for any $\calA$) or $\km$ (in particular $\ke$), we have that $\mathrm{H}^1(X, \calO_X) = 0$. \end{prop}
	
	\begin{proof}Since slc singularities are Du Bois (see \cite{kk} and \cite[Corollary 6.32]{singmmp}), $X$ has Du Bois singularities. Then $\mathrm{H}^1(X, \calO_X) = 0$ since $\mathrm{H}^i(X_b, \calO_{X_b})$ is constant in any flat family of varieties with Du Bois singularities (see \cite[Corollary 1.2]{kk}), and any $X$ arises as the special fiber of a flat family whose general fiber is a surface $X_\eta$ with $\mathrm{H}^1(X_\eta, \calO_{X_\eta}) = 0$. \end{proof}
	
	\begin{remark} We will see in Theorem \ref{thm:boundary} that the surfaces on the boundary of $\km$ (and thus also $\ke$) satisfy that $\omega_X \cong \calO_X$. Moreover, if $F$ is the marked fiber, then $2F$ is an ample Cartier divisor such that $(2F)^2 = 2$. Then following \cite[Definition 3.4, Proposition 3.8, and Theorem 3.11]{aet}, we see that $\km$ and $\ke$ are proper Deligne-Mumford stacks representing a functor over arbitrary base schemes. Due to subtleties with defining moduli spaces in higher dimensions, the remaining spaces follow the formalism developed in \cite{master} and thus correspond to Deligne-Mumford stacks representing functors only over \emph{normal} base schemes (see \cite[Section 2.2.2]{master} for more details). \end{remark}

	\subsection{Isotrivial $j$-invariant $\infty$ fibrations}\label{sec:jinfty}
	
	Here we prove some preliminary results on isotrivial $j$-invariant $\infty$ elliptic fibrations of K3 type which appear in the boundary of the various moduli spaces described above. We begin by bounding the number of $\mathrm{N}_i$ fibers (Definition \ref{def:nk}) which can appear on an slc elliptic K3.
	
	\begin{prop}\label{prop:k3typejinfty} Let $(f : X \to \mathbb{P}^1,S)$ be an isotrivial $j = \infty$ slc Weierstrass fibration of K3 type. Then $X$ has one of the following configurations of cuspidal fibers: $(1) \,   4\mathrm{N}_1, (2) \,   2\mathrm{N}_1\mathrm{N}_2,$ or $(3) \,  2\mathrm{N}_2$.
	\end{prop} 
	
	\begin{proof} We must have only $\mathrm{N}_0, \mathrm{N}_1$ and $\mathrm{N}_2$ by the slc assumption so by Proposition \ref{prop:jinftyfibers}, $2 = a_1/2 + a_2$ which only admits the non-negative integer solutions $(a_1,a_2) = (4,0), (2,1)$ and $(0,2)$. \end{proof}
	
	\begin{remark} Up to automorphisms of $\mathbb{P}^1$, the global Weierstrass equation for the surfaces in Proposition \ref{prop:k3typejinfty} can be written as follows: 
		\begin{enumerate}
			\item $y^2 = x^3 - \frac{1}{3}t^2s^2(t-s)^2(t-\lambda s)^2x + \frac{2}{27}t^3s^3(t - s)^3(t- \lambda s)^3$, $\lambda \in \mb{P}^1 \setminus \{0,1,\infty\}$, 
			\item $y^2 = x^3 - \frac{1}{3}t^2s^2(t - s)^4 x + \frac{2}{27}t^3s^3(t- s)^6$, 
			\item $y^2 = x^3 - \frac{1}{3}t^4s^4 x + \frac{2}{27}t^6s^6$. 
		\end{enumerate}
		In particular, up to isomorphism there is a unique surface with configuration $(2)$ and $(3)$. 
	\end{remark}
	
	Finally, we need the following key proposition.
	
	\begin{prop}\label{prop:Nkdiscriminant} Suppose $(f_0 : X \to \mb{P}^1, S)$ is an isotrivial $j = \infty$ slc Weierstrass fibration of K3 type and $F \subset X$ is an $\mathrm{N}_k$ fiber. If $f_0$ is the central fiber of a $1$-parameter family of Weierstrass models $(f : \mathscr{X} \to \mathscr{C}, \mathscr{S}) \to B$ with generic fiber $(f_\eta : \mathscr{X}_\eta \to C_\eta, \mathscr{S}_\eta)$ a $24\mathrm{I}_1$ elliptic fibration, then there are at least $k + 1$ type $\mathrm{I}_1$ fibers of $f_\eta$ that limit to the $\mathrm{N}_k$ fiber $F$ for $k = 1,2,3,4$.
	\end{prop}
	
	\begin{proof}Consider the twisted stable maps limit of $f_\eta$. By Proposition \ref{prop:tsm} (1), the Weierstrass $\mathrm{N}_1$ fiber $F$ must be replaced by a surface component $Y$ attached along the twisted model of $F$ by a twisted fiber of type $\mathrm{I}^*$ (resp. $\mathrm{I}$) if $k$ is odd (resp. even). By Proposition \ref{prop:k3typejinfty}, the possibilities for $X$ are $4\mathrm{N}_1$, $2\mathrm{N}_1\mathrm{N}_2$, $2\mathrm{N}_2$ as well as the non-slc cases $\mathrm{N}_1\mathrm{N}_3$ and $\mathrm{N}_4$. Since the degree of the $j$-map is constant for a family of twisted stable maps, the sum of degrees of the $j$-map of the components of the twisted model is $24$. This means that $Y$ is rational when $k = 1,2$ or K3 when $k = 3,4$. The number of $\mathrm{I}_1$ fibers of $f_\eta$ limiting to the $\mathrm{N}_1$ fiber $F$ of $f_0$ is the same as the number of $\mathrm{I}_1$ fibers limiting to the component $Y$ in the twisted model. By conditions $(2)$ and $(3)$ in \emph{loc. cit.}, the component $Y$ cannot be isotrivial and $\deg(\mathscr{L}) \geq 1$. By Persson's classification \cite{persson}, a rational elliptic surface $Y$ with an $\mathrm{I}^*$ fiber has $\geq 2\mathrm{I}_1$ fibers, and one with an $\mathrm{I}_n$ has $\geq 3$ other $\mathrm{I}_1$ fibers counted with multiplicity. Similarly, by \cite[Theorem 1.1 and 1.2]{shioda}, an elliptic K3 surface with an $\mathrm{I}^*$ fiber has $\geq 4\mathrm{I}_1$ fibers, and one with an $I_n$ fiber has $\geq 5$ other $\mathrm{I}_1$ fibers counted with multiplicity.   \end{proof}

	\section{Wall crossings inside $\ska$ for $a > \frac{1}{12}$}\label{sec:firstwall}
	
	Recall that $\ska$ denotes the space where all singular fibers are marked with weight $a$ and we have taken the $S_{24}$ quotient. The main goal of this section (see Section \ref{sec:explicitktwelve}) is to describe the surfaces parametrized by $\ktwelve$ for $0 < \epsilon \ll 1$. In particular, we explicitly describe the wall crossings that happen as we vary the weight vector from $a = 1$ to $a = 1/12 + \epsilon$.
	
	By Corollary \ref{cor:twocomponents} we see that surfaces parametrized by $\ska$ have \emph{at most} two elliptically fibered components, but possibly with trees of pseudoelliptic surfaces attached to them. In Proposition \ref{prop:singlenormal12} we classify the possible surfaces parametrized by $\ska$ with a \emph{single normal elliptically fibered component}. In Theorem \ref{thm:singlejinfty} we classify the possible surfaces parametrized by $\ska$ with a \emph{single non-normal elliptically fibered component}. In Theorem \ref{thm:twomaincomp}, we classify the possible surfaces 
	parametrized by $\ska$ with \emph{two elliptically fibered components}. Finally, in Proposition \ref{prop:exist1} and Proposition \ref{prop:exist2}, we show that surfaces of each type appearing in the aforementioned results do exist on the boundary of $\ska$.
	
	\begin{lemma}There are Type $\wii$ walls where Type $\mathrm{I}$ pseudoelliptic surfaces form at $a = \frac{1}{k}$ for $k = 1, \dots, 11$. \end{lemma}
	\begin{proof} Recall that Type I pseudoelliptic surfaces form when a component of the underlying weighted curve is contracted -- this occurs when $ka=1$. Finally, note that $24a>2$ for each of these values of $k$ so that the moduli space is nontrivial. \end{proof}
	
	\begin{lemma} There are Type $\wiii$ walls at
		$
		a = \frac{5}{12}, \frac{3}{12}, \text{ and } \frac{2}{12}
		$
		where rational pseudoelliptic surfaces attached along intermediate type $\mathrm{II}$, $\mathrm{III}$ and $\mathrm{IV}$ fibers respectively contract to a point. \end{lemma} 
	
	\begin{proof} This follows from \cite[Theorem 6.3]{master} as well as the observation that that a rational elliptic surface attached to a type $\mathrm{II}, \mathrm{III}$ or $\mathrm{IV}$ fiber must have a $\mathrm{II}^*, \mathrm{III}^*$, or $\mathrm{IV}^*$ fiber respectively and so it has $2,3$, or $4$ other marked fibers counted with multiplicity. \end{proof} 
	
	Since the above walls are all above $\frac{1}{12}$, we obtain the following: 
	
	\begin{cor}\label{cor:234:weierstrass} Any type $\mathrm{II}, \mathrm{III}$ and $\mathrm{IV}$ fiber on a surface parametrized by $\ktwelve$ is a Weierstrass fiber. In particular, there are no pseudoelliptic trees sprouting off of it. 
	\end{cor}
	
	In a similar vein we have the following two lemmas: 
	
	\begin{lemma}\label{lemma:N1} There are Type $\wiii$ walls at
		$
		a = \frac{1}{4}, \frac{1}{6}, \frac{1}{8}, \text{ and } \frac{1}{10}, \text{ where: }
		$
		\begin{enumerate}
			\item rational pseudoelliptic surfaces attached along intermediate type $\mathrm{N}_1$ fibers contract onto a point;
			\item isotrivial $j$-invariant $\infty$ surfaces with $\deg \mathscr{L}=1$ attached along intermediate type $\mathrm{N}_1$ fibers contract onto a point. 
		\end{enumerate}
	\end{lemma} 
	
	\begin{proof} A rational elliptic surface attached along an $\mathrm{N}_1$ fiber must have an $\mathrm{I}_k^*$ fiber in the double locus. Since an $\mathrm{I}_k^*$ has discriminant $6 + k$, then there are $6 - k$ markings counted with multiplicity on the rational pseudoelliptic. By the classification in \cite{persson}, there exist rational elliptic surfaces with $\mathrm{I}_k^*$ for $0 
		\leq k \leq 4$. Since the log canonical threshold of an intermediate $\mathrm{N}_1$ fiber is $\frac{1}{2}$, then the surfaces with $\mathrm{N}_1/\mathrm{I}_k^*$ double locus contract at $\frac{1}{2(6 - k)}$. These give walls above $\frac{1}{12}$ for $1 \leq k \leq 4$. Similarly, isotrivial $j$-invariant $\infty$ surfaces with an $\mathrm{N}_1$ fiber and $\deg \mathscr{L} = 1$ must be attached along another $\mathrm{N}_1$ fiber and so contract at $\frac{1}{2k}$ where they support $k$ fibers. 
	\end{proof}

	Next we consider the base curve at $\frac{1}{12} + \epsilon$. 
	
	\begin{lemma}\label{lem:stablecurves} Let $\calA = (a, \ldots, a)$ for $a = \frac{1}{12} + \epsilon$. Then curves $C$ parametrized by $\overline{\calM}_{0, \calA}$ are either
		\begin{enumerate}
			\item a smooth $\bP^1$ with 24 marked points, with at most $11$ markings coinciding, or
			\item the union of two rational curves, each with 12 marked points and at most $11$ markings coinciding.
		\end{enumerate}
	\end{lemma} 
	
	\begin{proof} If $C$ is a smooth $\bP^1$, since the total weight for any marking is $\leq 1$ we see that $\leq 11$ points can coincide.   If $C$ is the union of two rational curves, since each point is weighted by $\frac{1}{12}+ \epsilon$, and since each curve needs total weight $>2$ (including the node), each curve must have (exactly) 12 points, and again at most 11 can coincide. Finally, suppose $C$ is the union of three components $C = \cup_{i =1}^3 C_i$ with $C_1$ and $C_3$ the end components. Since the $C_2$ component needs at least one marking to be stable, at least one of $C_1, C_3$ will not have enough marked points to be stable. 
	\end{proof}

	\begin{cor}\label{cor:twocomponents} Let $(f : X \to C, S + F_a)$ be a surface pair parametrized by $\ktwelve$. Then $f : X \to C$ has at most two elliptically fibered components. \end{cor}
	
	\begin{remark} Note that  $X$ can have many Type I pseudoelliptic components mapping by $f$ onto marked points of $C$. \end{remark}
	
	\begin{definition}\label{def:maincomponent} If $(f : X \to C, S + F_a)$ a surface pair parametrized by $\ktwelve$, the \textbf{main component of $X$} denoted by $X_m$, is the union of all elliptically fibered components of $f : X \to C$. 
	\end{definition} 
	
	\begin{remark} \label{rem:maincomponent} By Corollary \ref{cor:twocomponents}, for all surfaces pairs parametrized by $\ktwelve$, either $X_m$ and $C$ are irreducible or $X_m = X_1 \cup X_2$ and $C = C_1 \cup C_2$ where $X_i$ and $C_i$ are irreducible $f|_{X_i} : X_i \to C_i$ is an elliptic fibration.  
	\end{remark}

	\subsection{Explicit classification of surfaces inside $\ktwelve$}\label{sec:explicitktwelve}
	
	We conclude that every surface parametrized by $\ktwelve$ consists of a main component (see Definition \ref{def:maincomponent}) possibly with trees of pseudoelliptics sprouting off. In order to do understand the possible main components $X_m$ parametrized by $\ktwelve$, we will use the following construction of a Weierstrass model for $X_m$. 
	
	\subsubsection{Construction of a family of Weierstrass models}\label{rmk:construction} 
	Let $(f_0 : X_0 \to C_0, S_0 + (F_a)_0)$ be an elliptic surface pair parametrized by $\ktwelve$, which by Corollary \ref{cor:twocomponents} has at most two elliptic components. Consider a 1-parameter family $(f : \mathscr{X} \to \mathscr{C}, \mathscr{S} + \mathscr{F}_a) \to T$ with generic fiber $(f : X_\eta \to C_\eta, S_\eta + (F_a)_\eta)$ a 24$\mathrm{I}_1$ elliptic K3 surface and special fiber $X_0$. Let $\mathscr{G}_\eta$ be a generic smooth fiber of the elliptic fibration $f: \mathscr{X} \to \mathscr{C}$ so that the closure $\mathscr{G}$ is a generic smooth fiber of $f_0 : X_0 \to C_0$. In particular, $G_0 = \mathscr{G}_0$ avoids any pseudoelliptic trees of $X_0$. 
	
	Let $Y_0$ denote the irreducible component of $X_0$ on which $G_0$ lies. The component $Y_0$ is necessarily elliptically fibered, and so either $Y_0 = X_m$ is the main component or $X_m = Y_0 \cup_{H_0} Y_1$ glued along a twisted fiber $H_0$. To classify the possible elliptically fibered components of $X_0$, we will take the relative log canonical model of the pair $(\mathscr{X}, \mathscr{S} + \mathscr{G}) \to T$ using the main results of \cite{master}. 
	
	First, if $X_m = Y_0 \cup Y_1$, there is a Type $\wii$ crossing causing a flip of the section of $Y_1$ so that $Y_1$ becomes a Type $\mathrm{I}$ pseudoelliptic. Then in either case, we have a new family where $Y_0$ is the unique elliptically fibered component with trees of Type $\mathrm{I}$ pseudoelliptic surfaces sprouting off of it. We make the following assumption, and revisit it when we see it holds in Lemmas \ref{lem:nonlc} and \ref{lem:sing:noass1}.
	
	\begin{assumption}\label{ass:1} Suppose every Type $\mathrm{I}$ pseudoelliptic tree attached to $Y_0$ is attached along the intermediate model of a log canonical Weierstrass cusp.
	\end{assumption}
	
	There exists a sequence of Type $\mathrm{W}_\mathrm{III}$ extremal contractions followed by a Type $\mathrm{W}_\mathrm{III}$ relative log canonical morphism of the family that contract the trees of Type $\mathrm{I}$ pseudoelliptic components to a point resulting in a Weierstrass model $Y'$ of $Y_0$. Denote the resulting family of surfaces $\mathscr{X}' \to T$. 
	
	Since Type $\mathrm{W}_\mathrm{III}$ contractions preserve the generic fiber of the family $\mathscr{X} \to T$, we must only check Type $\mathrm{W}_\mathrm{II}$ contractions of the section $S$. By \cite[Proposition 5.9]{giovanni}, we may blow up the point to which the section has contracted to preserve the generic fiber of the family, and so  we have that $\mathscr{X}'_\eta = \mathscr{X}$. The resulting family of fibrations $(\mathscr{X}' \to \mathscr{C}) \to T$ is a family of slc Weierstrass models over $\mathbb{P}^1$ with $\deg(\mathscr{L}) = 2$, generic fiber a $24\mathrm{I}_1$ elliptic K3, and special fiber $Y'$. By Remark \ref{rmk:lcusp}, we can conclude that $Y'$ is one of the following:
	\begin{list1}\leavevmode
		\begin{enumerate}
			\item a minimal Weierstrass elliptic K3 surface ($\deg\mathscr{L} = 2$), 
			\item a rational elliptic surface with a single type L cusp, or
			\item an isotrivial elliptic surface with two type L cusps and all other fibers stable. 
		\end{enumerate}
	\end{list1}

	By considering the discriminant of $\mathscr{X}' \to \mathscr{C}$ as a flat family of divisors on $\mathscr{C}$, we have the following key observation: 
	
	\begin{remark}\label{obs1} Suppose $Y' \to C_0$ is normal. The number of $\mathrm{I}_1$ fibers of the generic fibration $X_\eta \to C_\eta$ that collide onto a singular fiber $F$ of $Y' \to C_0$ is the multiplicity of $F$ in the discriminant of the Weierstrass model $Y' \to C_0$. \end{remark} 
	
	We can use this observation to constrain the possible components of the twisted stable maps limit of $(f : \mathscr{X}_\eta \to \mathscr{C}_\eta, \mathscr{S}_\eta + \mathscr{F})$. In this limit, the singular fibers $(f : \mathscr{X}_\eta \to \mathscr{C}_\eta)$ cannot collide since they are marked with coefficient one. Let $Y''$ be the unique component of a twisted model that maps birationally to the component $Y'$ in the above family of Weierstrass models. Then each connected component of the complement of $Y''$ is a tree of twisted surfaces that gets collapsed onto a fiber of $Y''$ by the sequence of flips and contractions that produce the Weierstrass model above. In particular the number of marked fibers on each tree of elliptic components sprouting off a fiber of $Y''$ is exactly the multiplicity of the resulting of the discriminant of the resulting singular fiber on the Weierstrass model $Y'$. 
	
	\begin{remark}\label{obs2}The type L cusps are the Weierstrass model of an intermediate fiber of type $\mathrm{I}_m$ for $m \geq 0$. Such fibers are not contracted until they have coefficient 0, and so any pseudoelliptic tree glued along a type $\mathrm{I}_m$ fiber will remain when lowering coefficients to any $\epsilon > 0$. \end{remark}
	
	Finally we revisit Assumption \ref{ass:1}. We first need the following characterization of intermediate models of non-log-canonical Weierstrass cusps. 
	
	\begin{lemma} \label{lem:nonlc} Suppose $X = X_0 \cup_G X_1$ is a smoothable broken elliptic surface that is the union of broken elliptic surfaces $X_i \to C_i$ where $C_i \cong \mb{P}^1$ and each $X_i$ has a unique main component. Let $X'$ be the result of the Type $\mathrm{II}$ pseudoelliptic flip of the section of $X_0$, so that the strict transform $X_0'$ is attached to $X_1'$ by an intermediate fiber $A \cup G$. Then $A \cup G$ is the intermediate fiber of an slc cusp if and only if $-S_0^2 \le 1$, where $S_0$ is the section of $X_0 \to C_0$. 
	\end{lemma} 
	
	\begin{proof} The question is local around a neighborhood of the flip. Therefore, we may assume that $X_0$ and $X_1$ are irreducible, so that there are no pseudoelliptic trees sprouting off either of them. On the component $X_1'$ we have the divisor $S_1 + aA + G$. Note that $G$ has coefficient one since it is in the double locus, and the coefficient $a$ is given by the sum of coefficients of marked fibers on $X_1'$. Then the Weierstrass model of $A \cup G$ inside $X_1'$ has log canonical singularities if and only if the $G$ contracts onto the Weierstrass model in the log canonical model of the pair $(X_1, S + G)$, i.e., when all the coefficients on $X_0'$ are $0$. Since the pair is smoothable, this occurs if and only if $X_0'$ contracts to a point in the log canonical model of $X$, where all the coefficients on $X_0'$ are set to $0$. Since $G$ is marked with coefficient one on $X_0'$, this occurs if only if $X_0'$ is a minimal rational elliptic surface by \cite[Proposition 7.4]{calculations} which holds if and only if $-S_0^2 \le 1$ (where the case $< 1$ happens if $G$ is a twisted fiber rather than a stable fiber of $X_0$). 
	\end{proof}

	\begin{lemma}\label{lem:sing:noass1} Let $X$ be a surface parametrized by $\ktwelve$ and suppose $Y \subset X_m$ is a normal main component. Then Assumption \ref{ass:1} is satisfied for every pseudoelliptic tree attached to $Y$. Moreover, the fibers these pseudoelliptic trees are attached to are minimal intermediate fibers. 
	\end{lemma} 
	\begin{proof} 
		Let $X' \to C'$ denote the twisted stable maps model of $X \to C$, and let $X_m'$ and $Y'$ denote the strict transform of $X_m$ and $Y$ in $X'$. Let $Z$ be a pseudoelliptic glued to an intermediate fiber $F$ of $Y$, and let $Z'$ be the components of $X'$ that map to $Z$. By Remark \ref{obs1}, the number of markings on $Z$ is equal to the contribution of $F$ to the discriminant of the Weierstrass model of $Y$. Since $X_m$ is the main component, there are   $< 12$ markings on $Z$, and so the order of vanishing of the discriminant of $F$ in $Y$ is $<12$. It follows that the order of vanishing of the Weierstrass data in a neighborhood of this fiber satisfy $\min\{3v(a), 2v(b)\} < 12$ so these are minimal Kodaira types by the standard classification. \end{proof}

	\subsubsection{$X_m$ is irreducible} We first deal with the case where the main component $X_m$ of a surface parametrized by $\ktwelve$ is irreducible. 
	
	\begin{prop}\label{prop:singlenormal12}
		Let $X$ be a surface parametrized by $\ktwelve$ such that the main component $X_m$ is irreducible and normal. Then $X_m$ is a minimal elliptic K3 surface with trees of pseudoelliptic surfaces of Type I attached along intermediate models of $\mathrm{I}_n^*, \mathrm{II}^*, \mathrm{III}^*$ and $\mathrm{IV}^*$ fibers. 
	\end{prop}
	
	\begin{proof} By Lemma \ref{lem:sing:noass1}, Assumption \ref{ass:1} is satisfied. Following Construction \ref{rmk:construction}, we saw that there are three possibilities for the Weierstrass stable replacement of the main component $X_m$ of a surface in $\ktwelve$. In case (1) we have a minimal Weierstrass elliptic K3 surface. Then since all fibers are minimal Weierstrass fibers, any pseudoelliptic surface has to be attached by the intermediate model of a minimal Weierstrass fiber. These are exactly the intermediate models of type $\mathrm{I}_n^*, \mathrm{II}, \mathrm{III}, \mathrm{IV}, \mathrm{II}^*, \mathrm{III}^*, \mathrm{IV}^*$, since type $\mathrm{I}_n$ Weierstrass fibers do not have intermediate models. By Lemma \ref{cor:234:weierstrass}, pseudoelliptics sprouting off of $\mathrm{II}, \mathrm{III}$ and $\mathrm{IV}$ fibers have contracted onto the Weierstrass model. We now rule out cases (2) and (3) of Construction \ref{rmk:construction}. 
		
		In case (2), the Weierstrass model of the main component is a rational elliptic surface with exactly one type L cusp. In this case, there must be a Type $\mathrm{I}$ pseudoelliptic tree $Z$ in $X$  attached to $X_m$ along an intermediate model of an L cusp, and by Remark \ref{obs1}, there are $12$ marked pseudofibers on $Z$. Let $X_1 \to C_1$ be a twisted stable maps model that maps to $X$ in $\ktwelve$. We may write $X_1 = Y_1 \sqcup_{\mathrm{I}_n} Z_1$ where
		\begin{enumerate}
			\item $Z_1$ is a broken elliptic fibration that dominates the pseudoelliptic tree $Z$,
			\item $Y_1$ is a broken elliptic fibration that dominates $X \setminus Z$, 
			\item  the component of $Y_1$ supporting the fiber $Y_1 \cap Z_1 = \mathrm{I}_n$ is birational to $X_m$, and
			\item the $Y_1 \cap Z_1 = \mathrm{I}_n$ fiber becomes the intermediate fiber on $X_m$ after $Z_1$ undergoes a Type $\mathrm{II}$ transformation into the pseudoelliptic tree $Z$. 
		\end{enumerate}
		Then $12$ of the marked fibers of $X_1 \to C_1$ must lie on $Z_1$ and the other $12$ on $Y_1$. In particular there is a node of $C_1$, such that if we separate $C_1$ along that node, we obtain two trees of rational curves each with $12$ marked points. However, this means the stable replacement of $C_1$ inside the Hassett space $\overline{\calM}_{0, \calA}$ for $\calA = (a, \ldots, a)$ with $a = \frac{1}{12} + \epsilon$, is a nodal union of two components, contradicting that $X$ has only one main component. 
		
		In case (3), the Weierstrass model of $X_m$ is a trivial surface with exactly two type L cusps and all other fibers stable. There must be Type $\mathrm{I}$ pseudoelliptic trees attached along each of these L cusp fibers in $X_m$, and no other pseudoelliptic trees attached to $X_m$, as every other fiber of its Weierstrass model is stable. As in the previous analysis, let $X_1 \to C_1$ be a twisted stable maps surface whose image in $\ktwelve$ is $X$, and let $X'$ be the component of $X_1$ that dominates $X_m$. Then $X'$ is attached to exactly two other components of $X_1$, so by stability it must have at least one marked point on it. Since $X_1 \to C_1$ is the twisted stable maps model, all the marked fibers have $j$-invariant $\infty$ and so since $X'$ is isotrivial, it must be non-normal, a contradiction. 
	\end{proof} 
	
	Next we consider the irreducible, but non-normal main component case.

	\begin{theorem}\label{thm:singlejinfty} Let $X$ be a surface parametrized by $\ktwelve$ with an irreducible non-normal main component $X_m$. Then one of the following holds. 
		
		\begin{enumerate}
			\item[(a)] $X_m$ is an isotrivial $j = \infty$ fibration with $4\mathrm{N}_1$ minimal Weierstrass fibers;
			\item[(b)] $X_m$ is an isotrivial $j = \infty$ fibration with $2\mathrm{N}_1$ minimal Weierstrass fibers, as well as an intermediate $\mathrm{N}_2$ fiber which must have a tree of pseudoelliptic surfaces attached to it along a type $\mathrm{I}_n$ pseudofiber;
			\item[(c)] $X_m$ is an isotrivial $j = \infty$ fibration with $2\mathrm{N}_2$ intermediate fibers each of which has a tree of pseudoelliptic surfaces attached to it by an $\mathrm{I}_n$ fiber;
			\item[(d)]$X_m$ is an isotrivial $j = \infty$ fibration with a minimal Weierstrass $\mathrm{N}_1$ fiber as well as an intermediate $\mathrm{N}_3$ fiber which has a tree of pseudoelliptic surfaces attached to it by an $\mathrm{I}_n^*$ fiber.  
			\item[(e)]$X_m$ is an isotrivial $j = \infty$ fibration with a single intermediate $\mathrm{N}_4$ fiber which has a tree of pseudoelliptic surfaces attached to it by an $\mathrm{I}_n$ fiber; or
		\end{enumerate}
		\noindent 
		
		Moreover, if we denote by $l$ the number of marked $\mathrm{N}_0$ fibers on $X_m$, then $l$ lies in the following range: (a): $4 \le l \le 16$, (b): $3 \le l \le 17$, (c): $2 \le l \le 18$, (d): $8 \le l \le 18$, (e): $13 \le l \le 19$.
	\end{theorem}
	
	\begin{proof} Suppose that Assumption \ref{ass:1} is satisfied. By Construction \ref{rmk:construction}, the Weierstrass model of the main component must be an slc isotrivial $j = \infty$ Weierstrass fibration with $\deg \mathscr{L} = 2$, which are classified by Proposition \ref{prop:k3typejinfty}. The lct of a type $\mathrm{N}_2$ fiber is $0$, so these do not contract to Weierstrass models, and any attached pseudoelliptic trees do not contract for nonzero weight. 
		
		In case $(c)$, the stability condition on the twisted stable maps limit implies that there must be at least one marked $\mathrm{N}_0$ to give that rational component of the base curve at least three special points. 
		
		The types of pseudofibers that are attached to intermediate $\mathrm{N}_1$ and $\mathrm{N}_2$ fibers respectively must have $j$-invariant $\infty$, so they are either type $\mathrm{I}_n$ or $\mathrm{I}_n^*$. The twisted model of an $\mathrm{N}_1$ fiber is a nonreduced rational curve, and so must have a stabilizer at the corresponding point of the twisted stable map. Therefore, it must be attached to an $\mathrm{I}_n^*$ fiber, which also has a nontrivial stabilizer at the corresponding point of the twisted stable map. Similarly, the twisted model of an $\mathrm{N}_2$ fiber is a nodal curve so it has no stabilizer, and therefore must be attached to an $\mathrm{I}_n$ fiber. 
		
		If Assumption \ref{ass:1} is not satisfied, then by Lemma \ref{lem:nonlc} we must have a K3 component $Y$ attached to $X_m$ along a fiber $F$ such that $Y$ is not the main component. This only happens if $Y$ has $<12$ singular fibers counted with multiplicity away from the fiber along which $Y$ is attached to $X_m$. In that case $F$ is a fiber of $Y$ with discriminant  $\geq 13$ so $F$ is either an $\mathrm{I}_n$ fiber for $n \geq 13$ or an $\mathrm{I}_n^*$ for $n \geq 7$. Consider a generic family of $24\mathrm{I}_1$ surfaces degenerating to this surface as in Section \ref{rmk:construction}.
		
		In the first case, we have that $n$ type $\mathrm{I}_1$ fibers collide to sprout out a trivial component with $n$ markings which becomes the main component when $Y$ flips into a pseudoelliptic. Since $X_m$ has only $\mathrm{N}_0$ fibers away from where $Y$ is attached and the degree of $\mathscr{L}$ must be $2$, then the attaching fiber is an $\mathrm{N}_4$ by Proposition \ref{prop:jinftyfibers}. This gives us (e). In the second case, let us denote by $Y'$ and $X_m'$ the strict transforms of $Y$ and $X_m$ in the twisted stable maps replacement of the limit of the family. Then $Y'$ and $X_m'$ are glued along twisted $\mathrm{I}_n^*/\mathrm{N}_1$ fibers since the order of the stabilizer is $2$. Then the base curve of the $X_m'$ component must have at least one more point with a stabilizer since any finite cover of $\mb{P}^1$ is ramified in at least two points. On the other hand, the stabilizer of any $j$-invariant $\infty$ curve is $\mu_2$ so these other points have to have stabilizers of order $2$. Now when the component $Y'$ flips into the pseudoelliptic surface $Y$, then the twisted fiber on $X_m'$ it is attached must flip into a \emph{non semi-log canonical} intermediate fiber since Assumption \ref{ass:1} fails. Thus it must be an $\mathrm{N}_k$ fiber for $k \geq 3$. The other twisted fibers on $X_m'$ must flip into intermediate models of $\mathrm{N}_k$ fibers for $k \geq 1$ since the $\mathrm{N}_0$ fiber has no stabilizers. Since the degree of $\mathscr{L}$ for the main component $X_m$ must be $2$, then by Proposition \ref{prop:jinftyfibers}, the fiber along which $Y$ is attached must be an $\mathrm{N}_3$ and the only other non-stable fiber is a single $\mathrm{N}_1$. This gives us case (d).  
		
		To obtain the number of markings, we may apply Proposition \ref{prop:Nkdiscriminant} to see that each $\mathrm{N}_k$ fiber is marked with multiplicity at least $k + 1$. This gives an upper bound on $n$. For the lower bound, we look at the largest number of marked $\mathrm{I}_1$ fibers that can appear on a component attached to the $\mathrm{N}_k$ fiber. For an $\mathrm{N}_1$ this is $5$ markings on an $5\mathrm{I}_1\mathrm{I}_1^*$ rational, for $\mathrm{N}_2$ this is $11$ markings on a $12\mathrm{I}_1$ (attached along one of the $\mathrm{I}_1$ fibers), for $\mathrm{N}_3$ this is $11$ markings on an $11\mathrm{I}_1\mathrm{I}_7^*$ elliptic K3, and for $\mathrm{N}_4$ this is $11$ markings on an $12\mathrm{I}_1\mathrm{I}_{13} $ elliptic K3. Here we have used that $X_m$ is the main component so all the other components must have undergone pseudoelliptic flips at a wall above $1/12 + \epsilon$. Finally, each $\mathrm{N}_1$ fiber is Weierstrass since there are at most $5$ markings on the component attached to it and so by Lemma \ref{lemma:N1}, these components contract to a point at a $\mathrm{W}_{\mathrm{III}}$ wall above $1/12 + \epsilon$.
	\end{proof} 
	
	\begin{remark} Each of the main components in Theorem \ref{thm:singlejinfty} that have only intermediate models of semi-log canonical cusps (e.g. cases (a), (b) and (c)) are $j = \infty$ limits of normal isotrivial elliptic surfaces. The $4\mathrm{N}_1$ surfaces are limits of $4\mathrm{I}_0^*$ isotrivial fibrations. Indeed, the locus in the moduli space of such surfaces is birational to $\mathbb{P}^1 \times \mathbb{P}^1$ where the first coordinate parametrizes the $j$-invariant of the fibration and the second coordinate parametrizes the configuration of the $4\mathrm{I}_0^*$ (respectively $4\mathrm{N}_1$) singular fibers. Similarly the $2\mathrm{N}_1 \mathrm{N}_2$ surface is the limit of the isotrivial $2\mathrm{I}_0^* \mathrm{L}$, surface and there is a rational curve of these in the moduli space. Finally the $2\mathrm{N}_2$ surface is the limit of isotrivial $2\mathrm{L}$ Weierstrass fibrations, but this family of $2\mathrm{L}$ surfaces does \emph{not} actually appear on this component of the moduli space as we describe below. 
		
		Note that in each of these cases, when the surface is isotrivial with $j \neq \infty$, all the markings must be concentrated on the special fibers. Indeed by Remark \ref{obs1}, there must be six markings concentrated at an $\mathrm{I}_0^*$ fiber and $12$ concentrated at a type $\mathrm{L}$ fiber. Therefore the isotrivial $j = \infty$ surface pairs that are limits of Weierstrass models as in the above paragraph must have six markings concentrated at each $\mathrm{N}_1$ fiber, and $12$ markings concentrated at each $\mathrm{N}_2$ fiber. In particular, they \emph{cannot} have any marked $\mathrm{N}_0$ fibers. Therefore, not all surface pairs with isotrivial $j = \infty$ main components are in the limit of the above locus of normal Weierstrass fibrations. In particular, since the type $2\mathrm{N}_2$ fibrations must have at least one marked $\mathrm{N}_0$ fiber by stability for twisted stable maps, we see that the $2\mathrm{L}$ family limiting to $2\mathrm{N}_2$ does not appear. 
	\end{remark}
	
	Finally we address the question of existence of each of the limits described above. 
	
	\begin{prop}\label{prop:exist1} Each of the cases described by Proposition \ref{prop:singlenormal12} and Theorem \ref{thm:singlejinfty} occurs in $\ktwelve$. 
	\end{prop}
	
	\begin{proof} We may take the Weierstrass model of the described main component. In each case it has a Weierstrass equation with $A,B$ of degree $8$ and $12$ respectively. Since the space of Weierstrass equations is irreducible, then there exists a family of $24\mathrm{I}_1$ elliptic $K3$ surfaces with this Weierstrass limit. By taking the stable replacement in $\ktwelve$ we must obtain stable limits as described.  \end{proof}

	\subsubsection{$X_m$ is reducible} Now we classify the broken elliptic surfaces in $\ktwelve$ where $X_m$ is the union of two irreducible surfaces.

	\begin{theorem}\label{thm:twomaincomp} Let $X$ be a surface parametrized by $\ktwelve$ with reducible main component, i.e. $X_m = Y_0 \cup Y_1$. Then one of the following holds.
		\begin{enumerate}
			\item $Y_i$ are rational elliptic surfaces glued along an $\mathrm{I}_0$ fiber. They are minimal Weierstrass surfaces away from possible intermediate Type $\mathrm{II}^*, \mathrm{III}^*$ and $\mathrm{IV}^*$ fibers along which Type $\mathrm{I}$ pseudoelliptic trees are attached.
			
			\item $Y_0$ is an elliptic K3 surface, $Y_1$ is a trivial $j$-invariant $\infty$ surface, and they are glued along $\mathrm{I}_{12}/\mathrm{N}_0$ fibers. There are $12$ marked $\mathrm{N}_0$ fibers on $Y_1$, and $Y_0$ has minimal Weierstrass fibers or minimal intermediate fibers Type $\mathrm{II}^*$, $\mathrm{III}^*$, or $\mathrm{IV}^*$ fibers where Type $\mathrm{I}$ pseudoelliptic trees are attached. 
			
			\item $Y_0$ is an elliptic K3 with an $\mathrm{I}_6^*$ fiber, $Y_1$ is an $2\mathrm{N}_1$ isotrivial $j$-invariant $\infty$ surface, and they are glued along twisted $\mathrm{I}_6^*/\mathrm{N}_1$ fibers. Away from the $\mathrm{I}_6^*$ fiber, $Y_0$ has minimal Weierstrass fibers or minimal intermediate $\mathrm{II}^*, \mathrm{III}^*$, and $\mathrm{IV}^*$ fibers where Type $\mathrm{I}$ pseudoelliptic trees are attached. There are $7 \le l \le 10$ marked $\mathrm{N}_0$ fibers on $Y_1$. 
			
			\item $Y_i$ are isotrivial $j$-invariant $\infty$ surfaces glued along $\mathrm{N}_0$ fibers. Each surface has a single intermediate $\mathrm{N}_2$ fiber with a Type $\mathrm{I}$ pseudoelliptic tree attached. There are $1 \le l_i \le 9$ marked $\mathrm{N}_0$ fibers on $Y_i$. 
			
			\item $Y_i$ are isotrivial $j$-invariant $\infty$ surfaces glued along $\mathrm{N}_0$ fibers. Each surface has $2$ minimal Weierstrass $\mathrm{N}_1$ fibers. There are $2 \le l_i \le 8$ marked $\mathrm{N}_0$ fibers on $Y_i$. 
			
			\item $Y_i$ are isotrivial $j$-invariant $\infty$ surfaces glued along $\mathrm{N}_0$ fibers. $Y_0$ has $2$ minimal Weierstrass fibers $\mathrm{N}_1$ fibers and $Y_1$ has one intermediate $\mathrm{N}_2$ fiber with a Type $\mathrm{I}$ pseudoelliptic tree attached. There are $2 \le l_0 \le 8$ marked $\mathrm{N}_0$ fibers on $Y_0$ and $1 \le l_1 \le 9$ marked $\mathrm{N}_0$ fibers on $Y_1$. 
			
		\end{enumerate}
		
	\end{theorem} 
	
	\begin{proof} We will proceed by taking the Weierstrass limit of the main component and using the classification in Section \ref{rmk:construction}  to determine what can be attached as the other main component. 
		
		First suppose that Assumption \ref{ass:1} does not hold for the fiber along which $Y_i$ are glued, so that after performing a pseudoelliptic flip of $Y_0$, the fiber on $Y_1$ is not the intermediate model of a semi-log canonical Weierstrass cusp. Then as in the proof of Theorem \ref{thm:singlejinfty}, $Y_0$ is a K3 component and $Y_1$ is an isotrivial $j$-invariant $\infty$ surface. Furthermore, they are either glued along twisted $\mathrm{I}_n/\mathrm{N}_0$ or $\mathrm{I}_n^*/\mathrm{N}_1$ fibers. Since they are the two main components, they must each have $12$ markings, so we conclude that $n = 12$ in the first case and $n = 6$ in the second case. Furthermore, as in the proof of Theorem \ref{thm:singlejinfty}, in the $\mathrm{I}_n^*/\mathrm{N}_1$ case, $Y_1$ must have another $\mathrm{N}_1$ fiber. This gives us cases (2) and (3) respectively.

		From now on we can suppose that Assumption \ref{ass:1} holds. Let us fix some notation. Denote the Weierstrass limit of $Y_i$ by $Y_i^0$ which must be one of the surfaces list in Section \ref{rmk:construction} if it is normal, or Proposition \ref{prop:k3typejinfty} if it is isotrivial $j$-invariant $\infty$. We will denote by $X^1 \to C^1$ a twisted stable maps model of the surface $X \to C$ in $\ktwelve$ and we will denote by $Y^1_i$ the unique component of $X^1$ dominating $Y_i$. Finally let $Z_i^1 \subset X^1$ the maximal connected union of connected components of $X^1$ that contains $Y_i^1$. Finally we will denote by $G$ the fiber along which $Y_0$ and $Y_1$ are glued, and by $G_i$ its model in the Weierstrass limit, which is obtained by flipping one of $Y_i$ and contracting the transform on $G$ on the other. See Figure \ref{fig:k31}.
		
		\begin{figure}[!h]
			\includegraphics[scale=.5]{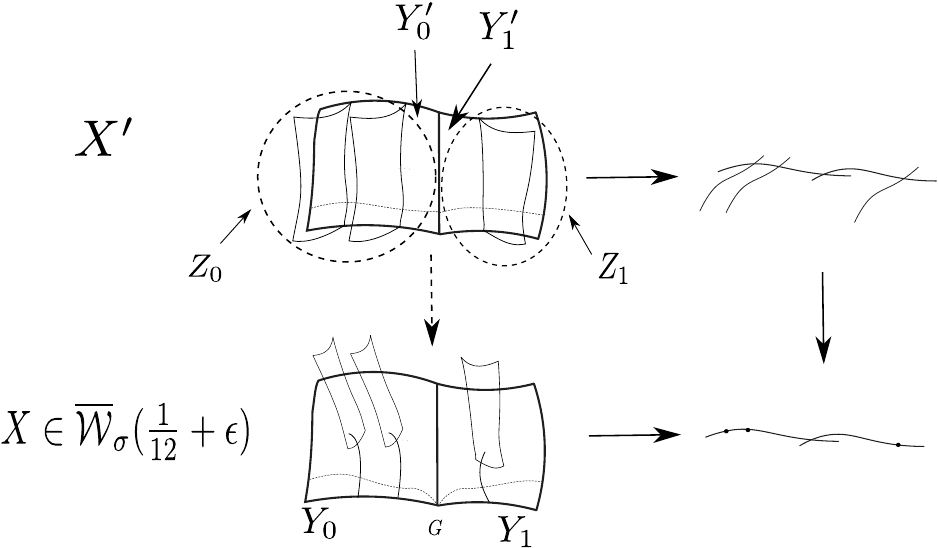}
			\caption{The circled components $Z_i$ represent the union of $Y_i^\prime$ along with the pseudoelliptic trees emanating from $Y_i^\prime$. The entire $Z_i$ component dominates $Y_i$, and the $Y_i^\prime$ components are the components containing the pseudoelliptics. }\label{fig:k31}
		\end{figure}

		Now since $Y_0$ and $Y_1$ satisfy Assumption \ref{ass:1} for the fiber along which they are glued, then by Lemma \ref{lem:nonlc} we must have $0 < -S_i^2 \le 1$ where $S_i$ are the sections of $Y_i$. Note that $S_0^2 \neq 0$, otherwise $Y_0$ would be \emph{trivial} and  so the degree the $j$-map on $Z_0$ would be $0$ and the degree of the $j$-map on $Z_1$ would be $24$, but then this would put us in situation (2).
		
		Suppose that $Y_0$ is normal. Then by Section \ref{rmk:construction}, $Y_0$ is a rational elliptic surface and $G_0$ is a type $\mathrm{L}$ cusp. Since the twisted model of a type $\mathrm{L}$ cusp is a stable curve, then $G$ is an $\mathrm{I}_n$ fiber. On the other hand, there must be $12$ markings on $Y_0$ away from $G$, and so $n = 0$ and $G$ is in fact a smooth fiber. Since $G$ is smooth, then $Y_1$ cannot be isotrivial $j$-invariant $\infty$ so it is normal and the same analysis applies to $Y_1$. Thus we obtain (1). 
		
		Next if $Y_0$ is not normal, then as above $Y_1$ is also non-normal. Now $Y_i$ satisfy Assumption \ref{ass:1} for the fiber $G$. We claim that they must also satisfy it for any pseudoelliptic trees away from $G$. Indeed suppose that $Y_0$ has an intermediate fiber $F$ not satisfying \ref{ass:1}. Then by Lemma \ref{lem:nonlc},  there must be an elliptic K3 attached to it. Every fiber of $Y_i$ is $\mathrm{N}_k$ for $k \le 2$ and we get  cases (4), (5), (6) by considering the various possible $\mathrm{N}_k$ fibers on a surface with $-S^2 \le 1$. 
		
		Since $\mathrm{N}_2$ fibers have $0$ lct, they must be intermediate with pseudoelliptic trees attached, while pseudoelliptic trees attached to an $\mathrm{N}_1$ fiber undergo type $\mathrm{W}_\mathrm{III}$ contractions at walls above $1/12 + \epsilon$ by Lemma \ref{lemma:N1} so $\mathrm{N}_1$ fibers are minimal Weierstrass. Finally, the number of markings is constrained by Proposition \ref{prop:Nkdiscriminant}, stability, and the fact that there are two main components so there must be $12$ total markings on each. \end{proof}
	
	\begin{prop}\label{prop:exist2} Each of the cases described in Theorem \ref{thm:twomaincomp} occurs in the boundary of $\ktwelve$. 
	\end{prop} 
	
	\begin{proof} Case (1) is the stable replacement in $\ktwelve$ of a Kulikov degeneration of Type II. Case (2) occurs when $12\mathrm{I}_1$ fibers collide to give an $\mathrm{I}_{12}$ fiber. Similarly, case (3) occurs when $12\mathrm{I}_1$ fibers collide to form an $\mathrm{I}_6^*$ fiber. Case (4) occurs when one starts with a degeneration of type (1) and take the limit as the $\mathrm{I}_1$ fibers approach the double locus $G$. Since marked $\mathrm{I}_1$ fibers from both $Y_0$ and $Y_1$ must fall into $G$ as the $j$-invariant of $G$ must match on both sides, then two isotrivial components appear so that each rational surface is attached to one of them along an $\mathrm{N}_0$ which leads to $\mathrm{N}_2$ fibers when the rational surfaces undergo a flip. Similarly, case (5) occurs when you start with a surface of type (1) and degenerate the two rational components into $2\mathrm{N}_1$ isotrivial $j$-invariant $\infty$ surfaces. Finally, for case (6), take a degeneration as in case (1) and then further degenerate $Y_0$ so that it is an isotrivial $2\mathrm{I}_0^*$ surface. Then the stable replacement of the limit as the $j$-invariant of the $2\mathrm{I}_0^*$ surface approaches $\infty$ is case (6).  \end{proof}

	\section{Surfaces in $\kepsilon$, the $24$-marked space at $a = \epsilon$}\label{sec:epsilon}
	
	In the previous section, we studied the wall crossings that occur in $\ska$ as we let the weight vary from 1 to $1/12 + \epsilon$, and we used this to classify the surfaces parametrized by the boundary of $\ska$ for $a = 1/12 + \epsilon$. The goal of this section is to explicitly study the wall crossings that occur as we reduce the weight further, from $a = 1/12 + \epsilon$ to $a = \epsilon$ for $0 < \epsilon \ll 1$. As a result, we determine the surfaces parametrized by the boundary of $\kepsilon$. The main results in this direction are Theorems \ref{thm:kepsilon} and Theorem \ref{thm:appear}. In Theorem \ref{thm:kepsilon} we describe the possible surfaces on the boundary, and in Theorem \ref{thm:appear} we use the theory of twisted stable maps (see Section \ref{sec:tsm}) to show that all such surfaces appear on the boundary. Finally, in Theorem \ref{thm:git2}, we describe a morphism from the coarse space of $\kepsilon$ to the GIT quotient $\oMG$. These three theorems together give a proof of Theorem \ref{thm:intro2}.\\
	
	We begin with the wall at $\frac{1}{12}$. 
	
	\begin{lemma}\label{lem:I*point} At $a = \frac{1}{12}$, there are Type $\mathrm{III}$ contractions of rational pseudoelliptic components attached by an $\mathrm{I}_0^*$ fiber. \end{lemma} 
	
	\begin{proof} An $\mathrm{I}_0^*$ must be attached along an other $\mathrm{I}_0^*$ by the stabilizer condition. Furthermore, an $\mathrm{I}_0^*$ rational surface has $6$ other markings with multiplicity. Putting this together with the description of the walls, we get a wall at $\frac{1}{2k} = \frac{1}{12}$ since $\frac{1}{2}$ is the lct of $\mathrm{I}_0^*$ (see Equation \ref{eq:a0} in Section \ref{sec:broken}). \end{proof} 
	
	\begin{lemma} At $a = \frac{1}{12}$ the trivial component $Y_1$ in case (2) of Theorem \ref{thm:twomaincomp} contracts onto the $\mathrm{I}_{12}$ fiber it is attached to. \end{lemma} 
	
	\begin{proof} The component of the base curve lying under $Y_1$ contracts to a point but since $Y_1$ is trivial, it contracts onto a fiber. \end{proof} 
	
	\begin{lemma} Let $X$ be a surface parametrized by $\ktwelve$ from Theorem \ref{thm:twomaincomp}(3). Then the stable replacement for coefficients $\frac{1}{12} - \epsilon$ is an irreducible pseudoelliptic K3 surface with an $\mathrm{I}_6^*$ fiber. \end{lemma} 
	\begin{proof} $X$ has main component $X_m = Y_0 \cup Y_1$ consisting of an elliptic K3 with a twisted $\mathrm{I}_6^*$ fiber glued to an isotrivial $j$-invariant $\infty$ surface along a twisted $\mathrm{N}_1$ fiber. Each surface has $12$ markings. At coefficient $\frac{1}{12} - \epsilon$, both section components are contracted by an extremal contraction. We first perform the extremal contraction of the section of $Y_1$ which results in a flip of $Y_1$ to a pseudoelliptic surface. Then the section of $Y_0$ contracts to form a pseudoelliptic with the pseudoelliptic model of $Y_1$ glued along an $\mathrm{I}_6^*$ pseudofiber. Finally, $Y_1$ contracts onto a point as in Lemma \ref{lem:I*point}.  \end{proof} 
	
	Putting the above together with the observation that the Hassett space becomes a point at $\frac{1}{12}$ so the base curves all contract to a point, we get the following:
	
	\begin{theorem}\label{thm:ktwelve}  Let $X$ be a surface parametrized by $\ktwelvee$.
		
		\begin{enumerate}
			\item If $X$ has a single main component, then $X_m$ is the pseudoelliptic surface associated to an elliptic surface as in Proposition \ref{prop:singlenormal12} and Theorem \ref{thm:singlejinfty} with an $A_1$ singularity where the section contracted. Any type $\mathrm{II}, \mathrm{III}, \mathrm{IV}$, $\mathrm{N}_1$ and $\mathrm{I}_k^*$ for $k \le 5$ pseudofibers of $X_m$ are Weierstrass and any $\mathrm{I}_n$ fibers satisfy $n \le 12$. There are pseudoelliptic trees sprouting off of intermediate Type $\mathrm{II}^*, \mathrm{III}^*, \mathrm{IV}^*$ and $\mathrm{N}_k$ for $k \geq 2$ fibers as before. 
			
			\item If $X$ has two main components, then $X_m$ is a union along a twisted pseudofiber of the surfaces appearing in Theorem \ref{thm:twomaincomp}, parts (1), (4), (5) and (6). Any type $\mathrm{II}, \mathrm{III}, \mathrm{IV}, \mathrm{N}_1$ and $\mathrm{I}_k^*$ for $k \le 5$ pseudofibers are Weierstrass. There are pseudoelliptic trees sprouting off of intermediate $\mathrm{II}^*, \mathrm{III}^*, \mathrm{IV}^*$ and $\mathrm{N}_2$ fibers as before. 
		\end{enumerate}
		
	\end{theorem}

	\begin{lemma} There are Type $\mathrm{III}$ walls at
		$
		a = \frac{1}{60}, \frac{1}{36}, \text{ and } \frac{1}{24}
		$
		where rational pseudoelliptic surfaces attached along intermediate type $\mathrm{II}^*$, $\mathrm{III}^*$ and $\mathrm{IV}^*$ fibers respectively contract to a point. \end{lemma} 
	
	\begin{proof} This follows from \cite[Theorem 6.3]{master} as well as the observation that that a rational elliptic surface attached to a type $\mathrm{II}^*, \mathrm{III}^*$ or $\mathrm{IV}^*$ fiber must have a $\mathrm{II}, \mathrm{III}$, or $\mathrm{IV}$ fiber respectively and so it has $10,9$, or $8$ other marked fibers counted with multiplicity. \end{proof}
	
	Next we study some examples of the transformations that occur for small coefficient.  
	
	\begin{example}\label{ex:1}(See Figure \ref{fig:k32})
		Suppose $X_{\eta}$ is a smooth elliptic K3 surface with 24 $(\mathrm{I}_1)$ fibers, and suppose it appears as the general fiber of a family $(f : \mathscr{X} \to B, \mathscr{S} + \mathscr{F}_a)$ with limit as in Theorem \ref{thm:singlejinfty} case (d). In particular, this is a stable limit for $a = \frac{1}{12} + \epsilon$ and $\mathscr{F}$ consisting of the $24\mathrm{I}_1$ fibers on the generic surface $X_\eta$. We will compute the stable limit of this family for $a < \frac{1}{12}$. We will denote by $X^{a}$ the $a$-stable special fiber of $\mathscr{X} \to B$. 
		
		We begin with the twisted stable maps limit $X^1\to C^1$. It consists of a union $Y_{0}^1 \cup Y_{1}^1$ where $Y_{0}^1$ is an elliptic K3 and $Y_{1}^1$ is a trivial $j$-invariant $\infty$ surface with $n$ marked fibers glued along an $\mathrm{I}_n$ fiber of $Y_{0}^1$ where $n > 12$. At $a = \frac{1}{24 - n}$, the component $Y_{0}^1$ undergoes a pseudoelliptic flip to obtain the model in Theorem \ref{thm:singlejinfty}  (d), i.e. $Y_0^a$ is a pseudoelliptic K3 glued along an intermediate $\mathrm{N}_4$ fiber $A^a \cup G^a$ of $Y_{1}^a$. Next, for $a \le \frac{1}{12}$, the section of $Y_1^a$ contracts onto an $A_1$ singularity so that $X^a$ consists of a pseudoelliptic isotrivial $j$-invariant $\infty$ surface with an intermediate $\mathrm{N}_4$ pseudofiber and a pseudoelliptic K3 sprouting off it. To continue the MMP on this $1$-parameter family and compute the stable limit for smaller $a$, we need to compute
		$
		(K_{\mathscr{X}^a} + \mathscr{F}^a).A^a$ and $(K_{\mathscr{X}^a} + \mathscr{F}^a).G^a.
		$
		
		We can restrict the log canonical divisor to the component $Y_1^a$ to obtain $K_{Y_1^a} + G + (24 - n)aA^a + naf$ where $f$ is a pseudofiber class. Pulling back to the blowup of the section $\mu : Y_1^{b} \to Y_1^a$ where $b = \frac{1}{12} + \epsilon$, we get
		$$
		\mu^*(K_{Y_1^a} + G + (24 - n)aA^a + naf^a) = K_{Y_1^b} + G^b + (24 - n)aA^b + naf^b + 12aS_1^b.
		$$
		Here $S_1^b$ is the section which is a $-2$ curve and $f^b$ is a fiber class. Now $A^b$ is the curve obtained by flipping the section $S_0$ of $Y_0^1$. Using the local structure of the flip (see e.g. \cite[Section 7.1]{ln}), we compute that $(A^b)^2 = -\frac{1}{2}$, $A^b.G^b = \frac{1}{2}$ and $(G^b)^2 = -\frac{1}{2}$. Similarly, using push-pull for the contraction $\rho : Y_1^b \to Y_1^1$ onto the twisted model of $Y_1^1$ we get that $K_{Y_1^b} = -2f^b + 2A^b$. Putting all these together and using push-pull for $\mu$ we get that 
		\begin{align*}
			(K_{Y_1^a} + G + (24 - n)aA^a + naf).A^a &= (K_{Y_1^b} + G^b + (24 - n)aA^b + naf^b + 12aS_1^b).A^b = \frac{na}{2} - \frac{1}{2} \\
			(K_{Y_1^a} + G + (24 - n)aA^a + naf).B^a &= (K_{Y_1^b} + G^b + (24 - n)aA^b + naf^b + 12aS_1^b).G^b \\ & = \frac{1}{2} + (24 - n)\frac{a}{2}.
		\end{align*}
		In particular, for $a < \frac{1}{n}$, there is an extremal contraction of the curve class of $A^a$ in $\mathscr{X}^a$. On the other hand, since $(A^b)^2 = -\frac{1}{2}$ and $\mu$ is the contraction of a $-2$ curve which intersects $A^b$ transversely, we have $(A^a)^2 = 0$ so this curve class rules $Y_1^b$ over $G^b$ and the extremal contraction for $a < \frac{1}{n}$ contracts $X^a$ onto $Y_0^a$, the pseudoelliptic K3. 
		
	\end{example}
	
	\begin{remark} We note that in the above example $n \leq 19$ by e.g.  \cite{shioda}. \end{remark}
	
	\begin{figure}[!h]
		\includegraphics[scale=.8]{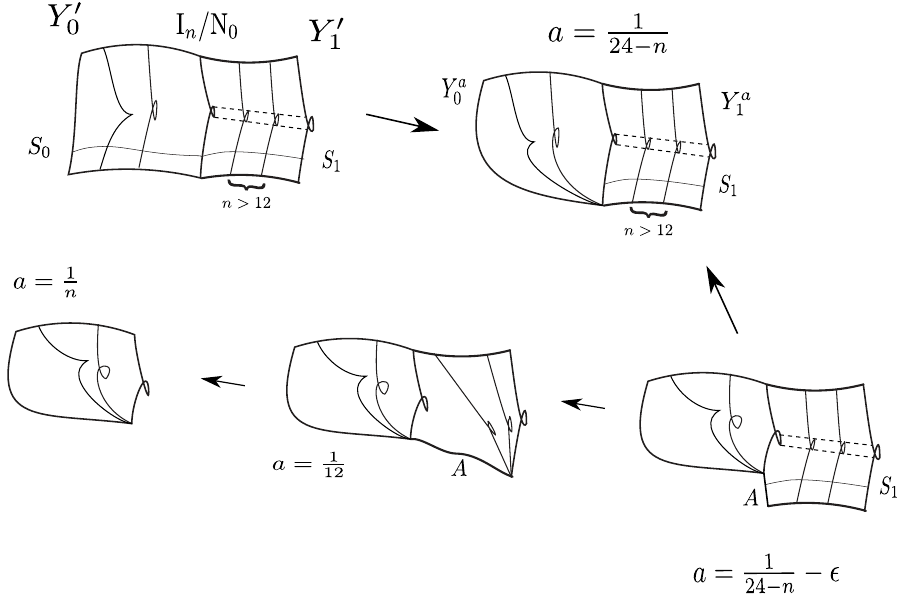}
		\caption{Illustration of Example \ref{ex:1}}\label{fig:k32}
	\end{figure}

	\begin{example}\label{ex:2}(See Figure \ref{fig:k33})
		Suppose $X_{\eta}$ as above is a smooth elliptic K3 surface with 24 $(\mathrm{I}_1)$ fibers which appears as the general fiber of a family $(f : \mathscr{X} \to B, \mathscr{S} + \mathscr{F}_a)$ with limit as in Theorem \ref{thm:singlejinfty} (e). We compute the stable limit for small $a$ as above and we keep the same notation. 
		
		The twisted stable maps limit $X^1\to C^1$ consists of a union $Y_{0}^1 \cup Y_{1}^1$ where $Y_{0}^1$ is an elliptic K3 and $Y_{1}^1$ is a $2\mathrm{N}_1$ isotrivial $j$-invariant $\infty$ surface. They are glued along twisted $\mathrm{I}_n^*/\mathrm{N}_1$ fibers with $n > 6$. At $a = \frac{1}{18 - n}$, the component $Y_{0}^1$ undergoes a pseudoelliptic flip to obtain the model in Theorem \ref{thm:singlejinfty} case (e), i.e. $Y_0^a$ is a pseudoelliptic K3 with a twisted $\mathrm{I}_n^*$ pseudofiber glued along an intermediate $\mathrm{N}_3$ fiber $A^a \cup G^a$ of $Y_{1}^a$. As above, the section of $Y_1^a$ contracts onto an $A_1$ singularity for $a \le \frac{1}{12}$  so that $X^a$ consists of a pseudoelliptic isotrivial $j$-invariant $\infty$ surface with an intermediate $\mathrm{N}_3$ pseudofiber and a pseudoelliptic K3 sprouting off it. The $\mathrm{N}_1$ pseudofiber of $Y_1^a$ may have a pseudoelliptic tree sprouting off of it, but it exhibits a Type $\mathrm{W}_\mathrm{III}$ contraction onto the Weierstrass model of the $\mathrm{N}_1$ fiber by Lemma \ref{lemma:N1}. 
		
		Restricting the log canonical divisor to the component $Y_1^a $, we obtain $K_{Y_1^a} + G + (18 - n)aA^a + (6+n)af$ where $f$ is a pseudofiber class. Pulling back to the blowup of the section $\mu : Y_1^{b} \to Y_1^a$ where $b = \frac{1}{12} + \epsilon$, we get
		$$
		\mu^*(K_{Y_1^a} + G + (18 - n)aA^a + (6 + n)af^a) = K_{Y_1^b} + G^b + (18 - n)aA^b + (6 + n)af^b + 12aS_1^b.
		$$
		As above, $A^b$ is the curve obtained by flipping the section $S_0$ of $Y_0^1$ which is a rational curve with self intersection $-\frac{3}{2}$ since $Y_0^1$ has a twisted $\mathrm{I}_n^*$ fiber. Thus we can compute that $(A^b)^2 = -\frac{2}{3}$, $A^b.G^b = \frac{1}{3}$ and $(G^b)^2 = -\frac{1}{6}$. Using push-pull for the contraction $\rho : Y_1^b \to Y_1^1$ onto the model of $Y_1^1$ with a twisted $\mathrm{N}_1$ for the double locus and a Weierstrass $\mathrm{N}_1$ for the other $\mathrm{N}_1$, we get that $K_{Y_1^b} = -f^b + A^b$. Putting all these together and using push-pull for $\mu$ we get that 
		\begin{align*}
			(K_{Y_1^a} + G + (18 - n)aA^a + (6 + n)af).A^a &= (K_{Y_1^b} + G^b + (18 - n)aA^b + (6 + n)af^b + 12aS_1^b).A^b \\ &= 
			\frac{2an}{3} - \frac{1}{3}  \\
			(K_{Y_1^a} + G + (18 - n)aA^a + (6 + n)af).B^a &= (K_{Y_1^b} + G^b + (18 - n)aA^b + (6 + n)af^b + 12aS_1^b).G^b \\ & = \frac{1}{6} + (18 - n)\frac{a}{3}.
		\end{align*}
		For $a < \frac{1}{2n}$, there is an extremal contraction of the curve class of $A^a$ in $\mathscr{X}^a$. On the other hand, since $(A^b)^2 = -\frac{2}{3}$ and $\mu$ is the contraction of a $-2$ curve which intersects $A^b$ transversely, we have $(A^a)^2 = -\frac{1}{6}$ so this curve class is rigid and therefore undergoes a flip. After the flip the strict transform $Y_1^a$ for $a < \frac{1}{2n}$ is now a pseudoelliptic attached along an intermediate pseudofiber of $Y_0^a$. By Lemma \ref{lem:nonlc}, the flipped pseudoelliptic contracts and goes through a Type $\mathrm{W}_\mathrm{III}$ pseudoelliptic flip for some small $a = \epsilon >0$ giving the stable limit as the minimal Weierstrass pseudoelliptic of $Y_0^a$. 
		
	\end{example}
	\begin{remark} By e.g. \cite{shioda}, the maximum $n$ such that there exists an elliptic K3 with an $\mathrm{I}_n^*$ is $14$ and so the above phenomena occur for $6 < n \le 14$.
	\end{remark}

	\begin{figure}[!h]
		\includegraphics[scale=.8]{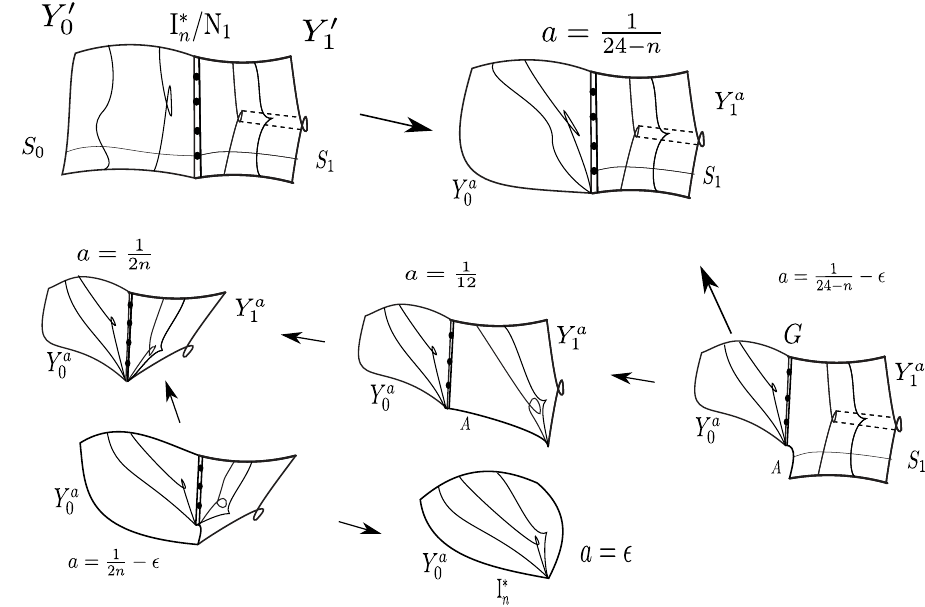}
		\caption{Illustration of Example \ref{ex:2}}\label{fig:k33}
	\end{figure}
	
	Combining these examples above we get the following:
	
	\begin{prop}\leavevmode \begin{enumerate}\item There are type $\mathrm{III}$ walls at $\frac{1}{k}$ for $13 \le k \le 19$ where the isotrivial $j$-invariant $\infty$ main component of the surfaces from Theorem \ref{thm:singlejinfty} case (d) contract as a ruled surface onto the $\mathrm{I}_n$ fiber of the pseudoelliptic K3 sprouting off of it. 
			
			\item  There are type $\mathrm{III}$ walls at $\frac{1}{2n}$ for $6 < n \le 14$ where the isotrivial $j$-invariant $\infty$ main component as in Theorem \ref{thm:singlejinfty} case (e) goes through a flip to become a pseudoelliptic attached to an intermediate model of the $\mathrm{I}_n^*$ on the $K3$ component. At some smaller $a = \epsilon > 0$, this pseudoelliptic contracts onto the Weierstrass model of the $\mathrm{I}_n^*$ fiber. 
		\end{enumerate}
	\end{prop}

	\begin{cor}\label{cor:stablereplacement} The stable replacement in $\kepsilon$ of the two main component surfaces of $\ktwelve$ from Theorem \ref{thm:twomaincomp} (d) and (e) is a pseudoelliptic K3 with a Weierstrass $\mathrm{I}_n$ respectively $\mathrm{I}_n^*$ fiber. \end{cor}

	\begin{prop}\label{prop:nomarkingcanonical}
		If $X$ is a surface parametrized by $\kepsilon$ then $\omega_X \cong \calO_X$.
	\end{prop}
	\begin{proof}
		
		If $X$ is irreducible then the result is clear since $X$ is the contraction of the section, a $(-2)$-curve, on a K3 type Weierstrass fibration. 
		
		Therefore, suppose $X$ is consists of multiple components. Let $p:\mathscr{X} \to D$ be a $1$-parameter family over the spectrum of a DVR with generic fiber a $24\mathrm{I}_1$ elliptic K3 and central fiber $X$. Now there is a sequence of pseudoelliptic flips producing a model $p':\mathscr{X}' \to D$ where the sections of $X$ are blown back up so that the components of central fiber $X'$ of $p'$ are all elliptically fibered and glued along twisted fibers (for example these flips occur as part of the MMP when decreasing the coefficient on the section of the twisted model, or equivalently, $X'$ is the model parametrized by the Brunyate/Inchiostro moduli space). Then $X' = X_0 \cup_{F_0} X_1, \ldots, \cup X_n \cup_{F_n} X_{n+1}$, where $X_0$ and $X_{n+1}$ are rational elliptic surfaces, and $X_1, \ldots, X_n$ are trivial $j$-invariant $\infty$ fibrations. 
		
		Then $K_{X'}|_{X_0} = K_{X_0} + F_0$, $K_{X'}|_{X_{n+1}} = K_{X_{n + 1}} + F_n$, and $K_X|_{X_i} = K_{X_i} + F_{i-1} + F_i$ for $i = 1, \ldots, n$ which are all $0$ by the canonical bundle formula since $X_0, X_{n+1}$ (resp. $X_1, \ldots, X_n$) satisfy $\deg \LL = 1$ (resp. $\deg \LL = 0$). Thus $K_{X'}$ is numerically trivial, that is, $K_{X'} \equiv 0$. 
		
		We proceed in two steps -- first we show that $X'$ is Gorenstein and then we show that the pullback
		\begin{equation}\label{eqpic}
			\Pic(X') \to \bigoplus_{i = 0}^{n+1} \Pic(X_i)
		\end{equation}
		is injective. For the first claim, note that away from the gluing fibers $F_i$, the surface $X'$ is a minimal Weierstrass fibration. From the classification of surfaces (e.g. see Corollary \ref{cor:stablereplacement}), the components $X_i$ are glued along $\mathrm{I}_n$ type fibers and so in a neighborhood of $F_i$, the surface corresponds to a map from a non-stacky nodal curve into $\overline{\calM}_{1,1}$. In particular, in a neighborhood of $F_i$, the elliptic fibration $X' \to C$ is a flat family of nodal curves over a nodal curve. In either case, $X'$ is Gorenstein. 
		
		Next denote by
		$
		\pi : \bigsqcup X_i \to X'
		$
		the natural morphism. By \cite[Proposition 2.6, Remark 2.7]{HP}  there is a diagram of short exact sequences of sheaves of abelian groups on $X'$
		$$
		\xymatrix{
			1 \ar[r] & \calO_{X'}^* \ar[r]^\alpha \ar[d] & \prod_{i = 0}^{n+1} \pi_*\calO_{X_i}^* \ar[r] \ar[d] & \calN \ar[r] \ar@{=}[d] & 0 \\
			1 \ar[r] & \calO_{F'}^* \ar[r]^\beta & \pi_*\calO_F^* \ar[r] & \calN \ar[r] & 0
		}
		$$
		where $F'$ is the double locus on $X'$ and $F$ is the double locus on $X_i$. Note that as an abstract variety, $F$ is the disjoint union of two copies of $F'$. By \cite[Proposition 4.2]{HP},  the map (\ref{eqpic}) is injective if and only if $\gamma:\Pic(F') \to \Pic(F)$ is injective and 
		$
		\mathrm{coker} H^0(\alpha) = \mathrm{coker} H^0(\beta). 
		$
		The map $\gamma$ is simply the diagonal so it is injective. Moreover, since $X'$, $X_i$ and $F_i$ are all connected projective varieties, taking $H^0$ of the above diagram gives
		$$
		\xymatrix{
			1 \ar[r] & k^* \ar[r]^{H^0(\alpha)} \ar[d]_{f_1} \ar[r] & \prod_{i = 0}^{n+1} k^* \ar[d]_{f_2} \\
			1 \ar[r] & \prod_{i = 0}^n k^* \ar[r]^{H^0(\beta)} & \prod_{i = 0}^n k^* \times k^*.
		}
		$$
		
		Here $f_1$ and $H^0(\alpha)$ are the diagonal maps, $H^0(\beta)$ is the product of diagonal maps for each $i$, and $f_2$ is given by
		$
		(x_0, \ldots, x_{n + 1}) \mapsto (x_0, x_1, x_1, x_2, \ldots, x_n, x_{n + 1}). 
		$
		The cokernel of $H^0(\alpha)$ can be identified with $\prod_{i = 1}^{n + 1} k^*$ by the map $(x_0, \ldots, x_{n + 1}) \mapsto (x_1/x_0, \ldots, x_{n + 1}/x_0).$
		Similarly, the cokernel of $H^0(\beta)$ can be identified with $\prod_{i = 0}^n k^*$ by the map
		$
		(a_0, b_0, a_1, b_1, \ldots, a_n, b_n) \mapsto (b_0/a_0, b_1/a_1, \ldots, b_n/a_n).
		$
		Putting this together, we see the induced map on cokernels is given by
		$
		(x_1, \ldots, x_{n + 1}) \mapsto$\\  $(x_1, x_2/x_1, \ldots, x_{n + 1}/x_n)
		$
		which is an isomorphism. Thus we conclude that (\ref{eqpic}) is an injection.

		Putting it all together, we have that $X'$ is Gorenstein and $\omega_{X'}$ pulls back to the trivial line bundle under (\ref{eqpic}) so $\omega_{X'} \cong \calO_{X'}$. It follows that $\omega_{\mathscr{X}'/D} \cong \calO_{\mathscr{X}'}$. Now $\mathscr{X}'$ is related to $\mathscr{X}$ by a sequence of log flips. Since these flips always contract $K$-trivial curves, we conclude from the Cone Theorem (e.g. \cite[Theorem 3.7 (4)]{km}) that the canonical is preserved so $\omega_\mathscr{X} \cong \calO_\mathscr{X}$ so $\omega_X \cong \calO_X$. \end{proof}

	Putting all of this together, we have a classification of the boundary components of $\kepsilon$ (see Section \ref{sec:explicitboundary} for an alternate description).

	\begin{theorem}\label{thm:kepsilon} The surfaces in $\kepsilon$  are the following:
		
		\begin{enumerate} 
			\item[(A)] An irreducible pseudoelliptic K3 with the section contracted to an $A_1$ singularity and minimal Weierstrass pseudofibers. 
			
			\item[(B)] An irreducible isotrivial $j = \infty$ pseudoelliptic with $4\mathrm{N}_1$ Weierstrass fibers.
			
			\item[(C)] An isotrivial $j = \infty$ fibration with $2\mathrm{N}_1$ Weierstrass fibers and an $\mathrm{N}_2$ intermediate fiber with a tree of pseudoelliptics sprouting off of it. 
			
			\item[(D)] An isotrivial $j = \infty$ fibration with $2\mathrm{N}_2$ intermediate fibers each sprouting a tree of pseudoelliptics.
			
			\item[(E)] A union of irreducible pseudoelliptic rational surfaces along an $\mathrm{I}_0$ fiber. 
			
			\item[(F)] A union of isotrivial $j = \infty$ pseudoelliptic surfaces with a single intermediate $\mathrm{N}_2$ fiber sprouting a pseudoelliptic tree on each, glued along an $\mathrm{N}_0$ fiber. 
			
			\item[(G)] A union of irreducible isotrivial $j = \infty$ surfaces each with $2\mathrm{N}_1$ Weierstrass fibers glued along an $\mathrm{N}_0$ fiber. 
			
			\item[(H)] A union of an irreducible isotrivial $j = \infty$ surface with $2\mathrm{N}_1$ Weierstrass fibers and an isotrivial $j = \infty$ surface with a single $\mathrm{N}_2$ fiber sprouting a pseudoelliptic tree, glued along an $\mathrm{N}_0$ fiber. 
			
		\end{enumerate}
		
		Furthermore, every surface $X$ satisfies $\omega_X \cong \calO_X$ and $\rm H^1(X, \calO_X) = 0$. Finally the number of marked $\mathrm{N}_0$ fibers are as in Theorem \ref{thm:singlejinfty} and Theorem \ref{thm:twomaincomp}.
	\end{theorem} 
	
	Now we show that each surface actually appears on the boundary, using the full smoothability results of Section \ref{sec:tsm}.
	
	\begin{theorem}\label{thm:appear}  Every slc surface pair in Theorem \ref{thm:kepsilon} appears in the boundary of $\kepsilon$.
	\end{theorem}
	\begin{proof} Given any surface satisfying the conditions of Theorem \ref{thm:kepsilon}, we can construct a twisted surface whose stable replacement is the surface obtained by flipping the pseudoelliptic components into elliptically fibered ones as in the previous section, replacing each cuspidal fiber by a twisted fiber, and attaching a component with dual monodromy satisfying the conditions of Propositions \ref{prop:tsm} \& \ref{prop:tsm2} to each of these twisted fibers. By the full smoothability Theorems \ref{thm:tsm} \& \ref{thm:tsm2}, this twisted model is the limit of a family of $24\mathrm{I}_1$ elliptic K3 surfaces with singular fibers marked and its stable replacement must be the initial surface as computed in the previous two sections. 
	\end{proof} 
	
	We conclude this section by discussion the connection between $\kepsiloncoarse$ and the GIT quotient $\oMG$.
	
	\begin{theorem}[Connection with GIT/SBB]\label{thm:git2}
		If $\kepsiloncoarse$ denotes the coarse space of $\kepsilon$, then there is a morphism $\kepsiloncoarse \to \oMG \cong \sbb$ with the following structure:
		\begin{enumerate}
			\item The locus of surfaces of type (A) maps isomorphically onto $\MG_s$.
			
			\item The locus of surfaces of type (B) maps as a generic $\mb{P}^{12}$-bundle onto $\oMG_{slc,o}$ by forgetting the marked fibers. The closure of this locus in $\kepsiloncoarse$ parametrizes the unique surface of type (G) along with a choice of marked fibers and this locus all maps onto $\oMG_{slc} \cap \oMG_L$.
			
			\item The locus of surfaces of type (E) maps onto $\oMG_L$ by taking the $j$-invariant of the $\mathrm{I}_0$ fiber along which the two components are glued. 
			
			\item The surfaces of type (C), (D), (F), and (H) all get mapped onto the point $\oMG_{slc} \cap \oMG_L$. 
			
		\end{enumerate}
	\end{theorem}
	
	\begin{proof} 
		By Theorem \ref{thm:kepsilon}, we have a classification of surfaces in $\kepsilon$. Each of the irreducible surfaces mentioned in the theorem are also parametrized by $\sbb$ yielding a rational map $\kepsiloncoarse \to \oMG$ defined on a dense open subset. Now one can easily check that the limit in $\oMG$ of a Weierstrass family limiting to a surface of type (B) (resp. of type (C), (D), (G), (F) and (H)) is the $j$-invariant of the L (resp. $\mathrm{N}_2$) fiber in $\oMG_{L}$. This depends only the central fiber of the family, not the family itself, so the morphism extends uniquely by normality after applying \cite[Theorem 7.3]{gg}. 
	\end{proof}
	
	\section{Explicit description of the boundary of $\kepsilon$}\label{sec:explicitboundary}
	In the previous section, specifically Theorems \ref{thm:kepsilon} and \ref{thm:appear}, we gave an explicit description of the surfaces parametrized by the boundary of $\kepsilon$. The goal of this section, is to enumerate the resulting boundary strata of $\kepsilon$ in a combinatorial way, and akin to Kulikov models (see Proposition \ref{prop:type2} for the analogue of \emph{Type II} degenerations, and Theorems \ref{thm:type31}, \ref{thm:type32}, and \ref{thm:type33} for the analogues of the \emph{Type III} degenerations).\\ 
	
	Before starting, we define $R_n$ to be the space parametrizing pairs $(X, S + F)$, where $X$ is a minimal Weierstrass rational elliptic surface, $S$ is a section, and $F$ is a fiber of type $\mathrm{I}_n$. Note $n \le 9$. The following is well known.
	
	\begin{lemma}\cite[Section 3.3]{hl} $R_n$ is a $9-n$ dimensional affine variety which is irreducible for $n \neq 8$ while $R_8$ has two components. 
	\end{lemma}
	
	Using these spaces, we will explicitly describe the boundary of $\kepsilon$. To do so, we use the notation of Kulikov models (i.e. type II and type III). 
	
	\subsection{Type II degenerations}
	We begin with the Type II degenerations. 
	
	\begin{prop}\label{prop:type2} There are two Type II strata described as follows. 
		
		\begin{enumerate}
			
			\item A dim. 17 stratum $\rm{W}_{II}$ isomorphic to a quotient of the fiber product $R_0 \times_j R_0$: namely the self fiber product of the $j$-map $j : R_0 \to \mb{A}^1$. A point parametrizes two rational elliptic surfaces with a marked $\mathrm{I}_0$ fiber of the same $j$-invariant glued along this fiber and the quotient comes from swapping the two surfaces ((E) in Theorem \ref{thm:kepsilon}). 
			
			\item A dim. $17$ stratum $\rm{W}_{II}^\infty \cong \Sym^{16}(\mb{P}^1) \times \mb{A}^1$ where $\mb{A}^1$ is the $j$-line. The $j$-line parametrizes the $4\mathrm{N}_1$ isotrivial $j$-invariant $\infty$ component and $\Sym^{16}(\mb{P}^1)$ parametrizes the $m$ markings on this surface other than the $\mathrm{N}_1$ fibers counted with multiplicity (Theorem \ref{thm:kepsilon} B).
			
		\end{enumerate}
		
	\end{prop}

	\subsection{Type III degenerations}
	We now discuss the type III degenerations. The first step is to ``un-flip'' the pseudoelliptic components in the description in Theorem \ref{thm:kepsilon}. After, we can describe each surface as a chain $X_0 \cup \ldots \cup X_{n+1}$, where both $X_0$ and $X_{n+1}$ are Weierstrass fibrations of rational type (i.e. $\deg \mathscr{L} = 1$), and $X_1, \ldots, X_n$ are all isomorphic to trivial $j$-invariant $\infty$ fibrations $C \times \bP^1$, with $C$ being a nodal cubic. These surfaces are all glued along nodal cubic fibers (i.e. either $\mathrm{I}_n$ or $\mathrm{N}_0$ fibers). Further, each $X_i$ for $i = 1, \ldots, n$ must have \emph{at least} one marked fiber by stability.  We call the surfaces $X_0$ and $X_{n+1}$ the \emph{end components} and $X_1, \ldots, X_n$ the \emph{intermediate components}. 
	
	\begin{lemma}\label{lem:endmarkings} An end component must have at least (a) $3$ marked fibers if it is normal, or (b) $4$ marked fibers if it is isotrivial $j$-invariant $\infty$, counted with multiplicity. \end{lemma} 
	
	\begin{proof} If an end component is an isotrivial $j$-invariant $\infty$ surface, then it must be $2\mathrm{N}_1$ fibration glued along an $\mathrm{N}_0$ fiber. Each $\mathrm{N}_1$ must carry at least $2$ markings counted with multiplicity so the surface carries at least $4$. If it is a normal rational elliptic surface, then the number of markings is given by $12 - n$ where the surface is glued along an $\mathrm{I}_n$ fiber. Since $n \leq 9$ for $I_n$ fibers on a rational elliptic surface, then there are at most $3$ markings on such a component. \end{proof} 
	
	\begin{cor} For the chains $X_0 \cup \ldots \cup X_{n+1}$ in the Type III locus, $n$ is at most $18$. \end{cor}
	
	\begin{proof} As there is $\geq 1$ marking on each of the intermediate components, the number of components is bounded by the number of markings not on $X_0$ and $X_{n+1}$. By Lemma \ref{lem:endmarkings}, there are $\geq 6$ combined on these components so $\leq 18$ markings to be distributed among the intermediate components. 
	\end{proof} 
	
	Now we will describe an explicit parameterization of each of the Type III strata. There are three cases depending on whether none, one, or both of the end components $X_0$ and $X_{n+1}$ are isotrivial $j$-invariant $\infty$. We call these strata of type $\mathrm{III}_0$, $\mathrm{III}_1$ and $\mathrm{III}_2$ respectively. The type $\mathrm{III}_0$ strata are further indexed by the fiber types $\mathrm{I}_{r}$ and $\mathrm{I}_{s}$ along which $X_0$ and $X_{n+1}$ are glued. In this case, there are $12 - r$ and $12-s$ fibers marked on $X_0$ and $X_{n+1}$ respectively which gives us $(r + s)$ markings remaining for the middle components $X_1, \ldots, X_n$. Thus, $n$ must satisfy 
	$
	1 \le n \le (r + s).
	$
	
	Finally, for each $n$, we can fix a single marking on each component $X_1, \ldots, X_n$ and fix coordinates so that the components are glued along fibers at $0, \infty$ and the chosen marking is at $1$. That gives us freedom to parametrize $r + s - n$ additional markings among the $X_1, \ldots, X_n$. For each choice of partition
	$
	\sum_{i = 1}^n a_i = r + s - n
	$
	we can consider the stratum where there are $a_i$ markings on $X_i$.
	
	\begin{theorem}[Type $\mathrm{III}_0$ locus]\label{thm:type31} Fix data
		$$ 1 \le r,s \le 9, \quad 1 \le n  \le r + s, \quad \sum_{i = 1}^n a_i = r + s - n, $$
		There is a type $A$ stratum
		$
		\mathrm{III}^{r,s,n}_{0, a_1, \ldots, a_n} 
		$
		of  $\dim(\mathrm{III}^{r,s,n}_{0, a_1, \ldots, a_n}) = 18 - n$ with a finite parameterization by 
		$
		R_s \times \mb{G}_m^{a_1} \times \ldots \times \mb{G}_m^{a_n} \times R_r.
		$
		Here a point of the above product determines the surface pairs $X_0, X_{n +1}$ as well as the configuration of $a_i$ marked fibers on the $X_1, \ldots, X_n$ avoiding the double locus. \end{theorem}
	
	\begin{remark} Just to reiterate, the $R_s$ and $R_r$ parametrize the surfaces $X_0$ and $X_{n+1}$ respectively, and the $\mb{G}_m^{a_i}$ parametrize the marked fibers on $X_i$ avoiding the double locus. \end{remark}
	
	Next, we consider type $\mathrm{III}_1$ strata where exactly one of the end surfaces, without loss of generality $X_0$, is an isotrivial $j$-invariant $\infty$ surface of rational type. Then $X_0$ must be the $2\mathrm{N}_1$ surface glued along an $\mathrm{N}_0$ fiber. There are $2$ markings each on the $\mathrm{N}_1$ fibers for a total of $4$. Then for each $0 \le s \le 17$, there is a stratum with $17 - s$ marked $\mathrm{N}_0$ fibers on $X_0$ (c.f. Theorem \ref{thm:singlejinfty}). After picking coordinates so that the $\mathrm{N}_1$ fibers are at $0$ and $1$ and the double locus is at $\infty$, these $17-s$ markings must avoid $\infty$ and so give a a factor of $\mb{A}^{17-s}$ parametrizing $X_0$. The other end component $X_{n+1}$ is a rational elliptic surface glued along an $\mathrm{I}_r$ fiber for some $r$ and with $12-r$ marked fibers.
	
	This gives $33 - s - r$ total markings on $X_0$ and $X_{n+1}$. On the other hand, there are at most $24$ markings so
	$
	33 - s - r \le 24. 
	$
	In the case of equality, there are no intermediate components and we have a stratum parametrized by $\mb{A}^{17 - s} \times R_r$. Otherwise, we have 
	$
	1 \le n \le s + r - 9
	$
	intermediate components with $s + r - 9$ markings distributed on them. After fixing one marking on each intermediate component at coordinate $1$, there are $r + s - 9 - n$ marked fibers partitioned into 
	$
	\sum_{i = 1}^n a_i = r + s - 9 - n. 
	$
	This gives a finite parameterization by 
	$
	\mb{A}^{17 - s} \times \mb{G}_m^{a_1} \times \ldots \times \mb{G}_m^{a_n} \times R_r
	$.
	
	\begin{theorem}[Type $\mathrm{III}_{1}$ locus]\label{thm:type32}\leavevmode
		\begin{enumerate}
			\item Fix the data
			$$  1 \le r\le 9, \quad
			0 \le s \le 17, \quad 
			s + r = 9
			$$
			There is a type $\mathrm{III}_1$ stratum $\mathrm{III}^{r,s}_{1}$ of $\dim(\mathrm{III}^{r,s}_{1}) = 17$ with a finite parameterization by $\mb{A}^{17 - s} \times R_r$. 
			\item Fix the data
			$$  1 \le r\le 9, \quad
			1 \le s \le 17, \quad
			1 \le n \le s + r - 9, \quad
			\sum_{i = 1}^n a_i = r + s - 9 - n.$$
			There is a type $\mathrm{III}_1$ stratum
			$
			\mathrm{III}^{r,s,n}_{1, a_1, \ldots, a_n}
			$
			of $\dim(\mathrm{III}^{r,s,n}_{1, a_1, \ldots, a_n}) = 17 - n$ with a finite pamaterization by 
			$
			\mb{A}^{17 - s} \times \mb{G}_m^{a_1} \times \ldots \times \mb{G}_m^{a_n} \times R_r.
			$
			
		\end{enumerate}
	\end{theorem}
	
	\begin{remark} Again, here $\mb{A}^{8-s}$ parametrizes the $8-s$ marked $\mathrm{N}_0$ on $X_0$, the $\mb{G}_m^{a_i}$ parametrizes the marked $\mathrm{N}_0$ on the $X_i$, and $R_r$ parametrizes the surface $X_{n+1}$. \end{remark}
	Finally, we have the type $\mathrm{III}_2$ stratum where both $X_0$ and $X_{n+1}$ are isotrivial $j$-invariant $\infty$. In this case both $X_0$ and $X_{n+1}$ are described by an affine space of dimension $17 - s$ and $17 - r$ respectively, where there are $17-s$ and $17-r$ marked $\mathrm{N}_0$ fibers on $X_0$ and $X_{n + 1}$ in addition to the $2\mathrm{N}_1$ which each appear with multiplicity $2$. This gives $42 - r - s$ total marked fibers among the end components, so $42 - r - s \le 24$ and we again have two cases: this is an equality and there are no intermediate components, or this inequality is strict and there are intermediate components with $r + s - 18$ marked fibers. Thus, as before so we obtain the following:
	
	\begin{theorem}[Type $\mathrm{III}_2$ locus]\label{thm:type33}\leavevmode
		\begin{enumerate} 
			\item Fix the data
			$$  0 \le s,r \le 17, \quad
			s + r = 18
			$$
			There is a Type $\mathrm{III}_2$ stratum $\mathrm{III}_{2}^{r,s}$ of  $\dim(\mathrm{III}_2^{r,s}) = 16$ with a finite parameterization by $\mb{A}^{17-s} \times \mb{A}^{17 - s} = \mb{A}^{16}.$
			\item Fix the data
			$$
			1 \le s,r \le 17, \quad
			1 \le n \le s+r-18 \quad
			\sum_{i = 1}^n a_i = r + s - n - 18.
			$$
			There is a Type $\mathrm{III}_2$ stratum 
			$\mathrm{III}^{r,s,n}_{2, a_1, \ldots, a_n}
			$
			of $\dim(\mathrm{III}^{r,s,n}_{2, a_1, \ldots, a_n}) = 16 - n$ with a finite parameterization by
			$
			\mb{A}^{17 - s} \times \mb{G}_m^{a_1} \times \ldots \times \mb{G}_m^{a_n} \times \mb{A}^{17 - r}.
			$
		\end{enumerate}
	\end{theorem}
	
	\begin{remark} In the above theorem, the $\mb{A}^{17-s}$ (resp. $\mb{A}^{17-r}$) parametrize the markings on $X_0$ (resp. $X_{n+1}$), and the $\mb{G}_m^{a_i}$ parametrize the markings on $X_i$. \end{remark}

	\section{The spaces with one marked fiber}\label{sec:onefiber}
	
	The goal of this section is to describe the surfaces parametrized by the boundary of the moduli spaces $\ke$ (resp. $\fe$), i.e. the moduli spaces parametrizing one $\epsilon$-marked singular fiber (resp. any fiber).
	In Section \ref{sec:onemarked} we describe the boundary of the two moduli spaces (see Theorem \ref{thm:boundary}). In Section \ref{sec:maptoGIT} we prove Theorem \ref{thm:git}, which describes a morphism from $\ke$ to $\oMG$.  Finally, in Section \ref{sec:git2} we extend Miranda's GIT construction to produce a moduli space of Weierstrass surfaces with a choice of marked fiber. The main result in this direction is Theorem \ref{thm:mainresult}, which shows that $\km$ is a smooth Deligne-Mumford stack with coarse space map $\km \to \oMGd$ given by the extended GIT compactification we discuss in Section \ref{sec:git2}.


	\subsection{The spaces with one marked fiber}\label{sec:onemarked}
	In this section we first consider the moduli space $\km$ (see Definition \ref{def:modulispace}), which corresponds to marking only one (possibly singular) fiber with $\epsilon$ weight. In particular, we give a description of the surfaces parametrized by the boundary. Note that since $\ke$ is a slice of $\km$, this description also applies to the surfaces parametrized by $\ke$. 
	
	\begin{theorem}[Characterization of the boundary]\label{thm:boundary} 
		The surfaces parametrized by $\km$ are single component pseudoelliptic K3 surfaces whose corresponding elliptic surfaces are semi-log canonical Weierstrass elliptic K3s, and the marked fiber $F$ can be any fiber other than an L type cusp. Moreover, all surfaces parametrized by $\km$ satisfy $\mathrm{H}^1(X, \calO_X) = 0$ and $\omega_X \cong \calO_X$. \end{theorem}
	
	\begin{proof}
		We follow the explicit stable reduction process explained in e.g. \cite[Section 6]{master}.  Let $(f: \mathscr{X} \to \mathscr{C}, \mathscr{S} + \mathscr{F}) \to T$ be a 1-parameter family whose generic fiber $(f: X_\eta \to C_\eta, S_\eta +  F_\eta)$ is a Weierstrass elliptic K3 surface with 24 $\mathrm{I}_1$ fibers, and a single (possibly singular) marked fiber $F_\eta$.  Denote by $(f_0: X_0 \to C_0, S_0 + F_0)$ the special fiber, and consider the limit obtained via twisted stable maps (see e.g. \cite{tsm}). The limit $(f_0 : X'_0 \to C'_0, S'_0 + F'_0)$, will be a tree of elliptic fibrations glued along twisted fibers, and the closure of the fiber $F$ will be contained in precisely one such surface component. While this surface will be stable as a map to $\overline{\calM}_{1,1}$, it will not necessarily be stable as a surface pair. To resolve this, pick some generic choice of markings $G = \cup_{i \in I} G_i$ to make the above limit stable as a surface pair. In this case, $G$ will consist of generic smooth fibers. 
		
		As we (uniformly) lower the coefficients marking $G$ towards 0, there will be some choice of coefficient so that the weighted stable base curve is an irreducible rational curve. Indeed, the components of the base curve will contract precisely when there is not enough weight being supported on the marked fibers. As we only lowered the coefficients marking $G$, and the fiber $F'_0$ remained marked with coefficient one, the (unique) main component, call it $Y_0$ fibered over the rational curve will contain the original marked fiber. 
		
		Now we have a single main component with marked fiber $F_0'$ with Type I pseudoelliptic trees attached to it. When the coefficients of $G$ are set to $0$ the Type I trees will undergo Type $\mathrm{W}_{\mathrm{III}}$ contractions to a point to produce the Weierstrass model of $Y_0$, away from the fiber $F_0'$. When the coefficient of $F_0'$ is reduced to $0 < \epsilon \ll 1$, it will cross $\mathrm{W}_{\mathrm{I}}$ walls to become a Weierstrass fiber.

		We saw in Proposition \ref{obvs:dubois} that $\rm{H}^1(X, \calO_X) = 0$, so it suffices to show that $\omega_X \cong \calO_X$. This holds on any Weierstrass elliptic K3 surface (see \cite[Proposition III.1.1]{mir3}), and since $X$ is obtained from a Weierstrass elliptic K3 by contracting a $(-2)$ curve (the section), we have $\omega_X \cong \calO_X$. 
	\end{proof}
	

	\subsection{Stable pairs to GIT / SBB}\label{sec:maptoGIT}
	The goal of this section is to describe the morphism from $\kepsiloncoarse \to \oMG$ (and thus to $\sbb$).

	\begin{theorem}[Connection with GIT / SBB]\label{thm:git}
		Let $\me$ be the coarse moduli space of $\ke$ and let $\Delta \subset \me$ be the boundary locus parametrizing surfaces with an L type cusp, with $U = \me \setminus \Delta$. There is a morphism $\me \to \oMG \cong \sbb$, such that the following diagram commutes: $$
		\begin{tikzcd}
			\Delta \arrow[r,hook] \arrow{d}[swap]{j} & \me \arrow[d] & U \arrow[d] \arrow[l,hook'] \\ 
			\mb{P}^1 \arrow[r] & \oMG & \oMG_s \arrow[l,hook']
		\end{tikzcd}
		$$
		
		$j: \Delta \to \mb{P}^1$ sends a surface with an L cusp to its $j$-invariant, the morphism $U \to \oMG_s$, is proper and finite of degree $24$, and $\mb{P}^1 \to \oMGL \subset \oMG$  maps bijectively onto the strictly GIT semistable locus. \end{theorem}

	\begin{proof}
		By Theorem \ref{thm:boundary} every surface parametrized by $\ke$ is a single component pseudoelliptic K3 surface. In particular, if we blow up the point to where the section contracted, we obtain an (unstable) slc Weierstrass elliptic K3 surface.  Consider the $\mathrm{PGL}_2$-torsor:
		$\calP = \{(X, s, t) \mid (s,t) \in C \cong \bP^1 \}/\sim,$ where $X$ is an slc Weierstrass elliptic K3 surface obtained by blowing up the section of a surface parametrized by $\ke$, the $(s,t)$ are coordinates on the base $C \cong \bP^1$ (or equivalently a basis for the linear series $|F|$ of a fiber $F$ on $X$), and we quotient by scaling.  Note that the Weierstrass coefficients $(A(s,t),B(s,t))$ defining $X$ are unique up to the scaling of the $\G_m$ action $(A,B) \mapsto (\lambda^4A, \lambda^6B)$.
		
		Since the semi-log canonical Weierstrass elliptic K3 surfaces are GIT semistable (\cite[Proposition 5.1]{mir}), we obtain a $\mathrm{PGL}_2$-equivariant morphism $\calP \to V$ which induces a morphism $\phi: \ke \to \oMG$. \end{proof} 
	
	\begin{remark}\leavevmode
		\begin{enumerate}
			\item The morphism $\ke \to \oMG$ is generically a 24 to 1 cover, as it requires the choice of some marked fiber, and generically there are 24 choices. The morphism is \emph{not} finite -- e.g. families with one L type cusp of fixed $j$-invariant are all collapsed to the same polystable point. 
			
			\item All the underlying surfaces of pairs parametrized by $\ke$ are in fact GIT semi-stable, even though all pairs with an L type cusp of fixed $j$-invariant map to the same GIT polystable point. One might wonder if the locus inside the GIT stack $[V_{24}^{ss} \sslash PGL_2]$ consisting of those surfaces that appear in $\ke$ is an open Deligne-Mumford substack with proper coarse moduli space factoring the morphism $\ke \to \oMG$. Furthermore, it is natural to compare this to a Kirwan desingularization of $\oMG$. We will pursue these questions in the future. 
			
			\item In the morphism from stable pairs to GIT, all surfaces with an L type cusp get collapsed to the polystable orbit corresponding to the KSBA-unstable, but GIT semistable (unique) surface with 2L cusps of the same $j$-invariant.
			\item The locus of surfaces with an L type cusp is 9 dimensional. Indeed, such surfaces are birational to a rational elliptic surface (which have an 8 dimensional moduli space) with a choice of a fiber to replace by an L type cusp. There is a $\mb{P}^1$ worth of choices. 
		\end{enumerate}
	\end{remark}

	\subsection{GIT for Weierstrass surfaces with a marked fiber}\label{sec:git2}
	
	We extend Miranda's GIT construction to produce a moduli space of Weierstrass surfaces with a choice of marked fiber. Such data can be represented by triples $(A,B,l)$ where $(A,B) \in V_{4N} \oplus V_{6N}$ are Weierstrass data as above, and $l \in V_1$ is a linear form. Then $\mb{G}_m \times \mb{G}_m \times SL_2$ acts naturally on $V_{4N} \oplus V_{6N} \oplus V_1$ where the first $\mb{G}_m$ acts on $V_{4N} \oplus V_{6N}$ with weights $4N$ and $6N$ and the second copy acts on $V_1$ with weight one. 
	
	To study GIT (semi-)stability, we follow Miranda's strategy. Consider the natural morphism 
	$
	f: V_{4N} \oplus V_{6N} \to S^3V_{4N} \oplus S^2V_{6N},
	$
	let $Z_N$ be the image of $f$, and let $\mathfrak{M}_N \subset \mb{P}(S^3V_{4N} \oplus S^2V_{6N})$ be  its projectivization.  The following proposition follows from \cite[Propositions 3.1 \& 3.2]{mir}:
	
	\begin{prop} The morphism 
		$
		f \times id : V_{4N} \oplus V_{6N} \oplus V_1 \to S^3V_{4N} \oplus S^2V_{6N} \oplus V_1
		$
		is finite and $SL_2$-equivariant with fibers contained in $\mathbb{G}_m \times \mathbb{G}_m$ orbits. In particular, two triples $(A,B,l)$ and $(A',B',l')$ are in the same $\mathbb{G}_m \times \mathbb{G}_m \times SL_2$ orbit if and only if the corresponding points in $\mathfrak{M}_N \times \mb{P}(V_1)$ are in the same $SL_2$ orbit. 
	\end{prop}
	
	This allows us to compute a GIT compactification of the moduli space of minimal Weierstrass fibrations with a chosen marked fiber as a GIT quotient $(\mathfrak{M}_N \times \mb{P}^1) \sslash SL_2$. We will linearize the moduli problem using the Segre embedding of $\mb{P}(S^3V_{4N} \oplus S^2V_{6N}) \times \mb{P}^1$. 
	
	\begin{prop}\label{prop:git} A triple $(A,B,l)$ is stable if and only if it is semi-stable. Furthermore, it is not stable if and only if there exists a point $q \in \mb{P}^1$ with 
		$
		v_q(A) > 2N \text{ and } v_q(B) > 3N
		$
		or with 
		$
		v_q(A) \geq 2N, \ v_q(B) \geq 3N, \text{ and } v_q(l) = 1
		$
		and at least one an equality. 
	\end{prop} 
	
	\begin{proof} Let $(A,B,l) \in \mathfrak{M}_N$ and let $\lambda: \mb{G}_m \to SL_2$ be a $1$-parameter subgroup and pick coordinates $[T_0, T_1]$ so that $\lambda$ acts by $T_0 \mapsto \lambda^e T_0$ and $T_1 \mapsto \lambda^{-e}T_1$. Then it acts on $(A,B,l)$ by
		\begin{align*}
			A &= \sum_{i = 0}^{4N} a_i T_0^i t_1^{4N - i} \mapsto \sum_{i = 0}^{4N} a_i \lambda^{2ei - 4eN} T_0^i t_1^{4N - i}, \quad 
			B = \sum_{i = 0}^{6N} b_i T_0^i t_1^{6N - i} \mapsto \sum_{i = 0}^{4N} b_i \lambda^{2ei - 6eN} T_0^i t_1^{4N - i} \\
			l &= l_0T_1 + l_1T^0 \mapsto l_0\lambda^{-e}T_1 + l_1 \lambda^eT_0.
		\end{align*}
		
		The coordinates of $\mb{P}(S^3V_{4N}\oplus S^2V_{6N}) \times \mb{P}(V_1)$ are given by $l_0a_ia_ja_k, l_0b_lb_m$, $l_1a_ia_ja_k$, and $l_0b_lb_m$ which  respectively have weights 
		$$
		2e(i + j + k) - 12eN - e, \quad
		2e(l + m) - 12eN - e, \quad
		2e(i + j + k) - 12eN + e,  \text{ and } \quad
		2e(l + m) - 12eN + e.$$
		
		By the Hilbert-Mumford criterion,  a point is not (semi-)stable if and only if there exists a $1$-parameter subgroup such that all the weights are non-negative (respectively positive).
		
		Suppose $(A,B,l)$ is not (semi-)stable and pick a $1$-parameter subgroup and coordinates as above. Then we have, after dividing by $e \neq 0$,
		\begin{align*}
			2e(i + j + k) - 12eN - e < (\le) \ 0 &\implies l_0a_ia_ja_k = 0, \,
			2e(l + m) - 12eN - e < (\le) \ 0 \implies l_0b_lb_m = 0\\
			2e(i + j + k) - 12eN + e < (\le) \ 0 &\implies l_1a_ia_ja_k = 0, \,
			2e(l + m) - 12eN + e < (\le) \ 0 \implies l_1b_lb_m = 0.
		\end{align*}
		Note that the left hand side is always odd and so equality is never achieved. From this we can conclude that stability coincides with semi-stability.  Now consider the cases where $i = j = k$ and $l = m$. Then we see that $l_0a_i^3 = 0$ for $i \le 2N$, $l_1a_i^3 = 0$ for $i \le 2N - 1$, $l_0b_l^2 = 0$ for $l \le 3N$, and $l_1b_l^2 = 0$ for $l \le 3N - 1$. Let $q = [0,1]$ be the point given by $T_0 = 0$. If $l_0 \neq 0$, then we must have that $a_i = 0$ for $i \le 2N$ and $b_l = 0$ for $i \le 3N$. Thus the order of vanishing $v_q(A) > 2N$ and $v_q(B) > 3N$.  Otherwise, if $l_0 = 0$ then $l_1 \neq 0$ so we must have that $a_i = 0$ for $i \le 2N - 1$ and $b_l = 0$ for $i \le 3N - 1$. In this case, $v_q(l) = 1$, $v_q(A) \geq 2N$ and $v_q(B) \geq 3N$. 
		
		Conversely, given a triple $(A,B,l)$ satisfying such order of vanishing conditions, we may pick coordinates such that $q = [0,1]$. Then it is easy to see that the $1$-parameter subgroup acting by $(T_0, T_1) \mapsto (\lambda T_0, \lambda^{-1}T_1)$ demonstrates that $(A,B,l)$ is not stable.  \end{proof} 
	
	In the case of K3 surfaces where $N = 2$, we obtain an especially pleasant result: 
	
	\begin{cor}\label{cor:git} A point of $\mathfrak{M}_2$ is stable if and only if it represents a $1$-marked Weierstrass fibration $(f : X \to \mb{P}^1, S + \epsilon F)$ with at worst semi-log canonical singularities. \end{cor} 
	
	\begin{proof} First note that the generic fiber of the fibration $f : X \to \mb{P}^1$ represented by a stable point in $\mathfrak{M}_N$ is at worst nodal since the Weierstrass data of a stable point cannot be identically $0$.  Then combining the above Proposition \ref{prop:git} with \cite[Lemma 3.2.1, Lemma 3.2.2, Corollary 3.2.4]{ln}, and noting that the log canonical threshold of a type $\mathrm{L}/\mathrm{N}_2$ fiber is 0 (see Lemma \ref{lem:lcusp}),  a point is unstable if and only if there exists a point $q \in \mb{P}^1$ such that the pair $(X, S + \epsilon F)$ is not semi-log canonical around the singular point of $f^{-1}(q)$. The result then follows since a Weierstrass fibration $(X, S + \epsilon F)$ has semi-log canonical singularities away from the singular points of the fibers. \end{proof}
	
	\begin{definition}\label{def:git} If $\mathfrak{M}_2^s$ denotes the stable/semi-stable locus, we denote  $\oMGd = \mathfrak{M}_2^s \sslash \SL_2$.\end{definition}

	\begin{theorem}\label{thm:mainresult} $\km$ is a smooth Deligne-Mumford stack with coarse space map $\km \to \oMGd$ given by the GIT compactification. Furthermore, there is a morphism $\km \to \oMG$ given by forgetting the marked fiber. A Weierstrass fibration $(f : X \to \mb{P}^1, S)$ is represented by a point in $\oMG$ if and only if there exists a fiber $F$ so that $(X, S + \epsilon F)$ is a stable pair. \end{theorem} 
	
	\begin{proof} By the proof of Theorem \ref{thm:git}, we obtain a birational morphism $\km \to [\mathfrak{M}_2^s/\PGL_2]$. On the other hand, by the above Corollary \ref{cor:git}, there is a family of  KSBA-stable one $\epsilon$-marked Weierstrass fibrations $(f : X \to \mb{P}^1, S + \epsilon F)$ over $\mathfrak{M}_2^s$. This induces a $\PGL_2$ equivariant map $\mathfrak{M}_2^s \to \km$ which gives an inverse map $[\mathfrak{M}_2^s/\PGL_2] \to \km$ exhibiting these as isomorphisms. On the other hand, $[\mathfrak{M}_2^s/\PGL_2]$ is a smooth stack as $\mathfrak{M}_2^s$ is an open subset of a smooth variety so $\km$ is smooth.
		
		The composition $\km \to [\mathfrak{M}_2^s/\PGL_2] \to \mathfrak{M}_2 \sslash \SL_2$ is the coarse moduli space map. Indeed, $[\mathfrak{M}_2^s/\SL_2]$ and $[\mathfrak{M}_2^s/\PGL_2]$ have the same coarse moduli space; note that $[\mathfrak{M}_2^s/\SL_2] \to [\mathfrak{M}_2^s/\PGL_2]$ is a $\mu_2$-gerbe being the base change of the map $B\SL_2 \to B\PGL_2$ so $[\mathfrak{M}_2^s/\SL_2] \to [\mathfrak{M}_2^s/\PGL_2]$ is a relative coarse space and the coarse map $[\mathfrak{M}_2^s/\SL_2] \to \mathfrak{M}_2^s \sslash \SL_2$ factors through it. 
		
		If $(A,B,l)$ is in $\mathfrak{M}^s_2$, then $(A,B)$ is a semi-stable point for Miranda, and conversely, if $(A,B)$ is semi-stable in Miranda's space then for a generic choice of fiber $F$, the corresponding fibration $(X \to \mb{P}^1, S + eF)$ is a stable pair and the corresponding GIT data $(A,B,l)$ is GIT stable. \end{proof}

	\bibliographystyle{alpha}
	\bibliography{K3}

\begin{thebibliography}{BHPVdV04}

\bibitem[AB17]{calculations}
Kenneth Ascher and Dori Bejleri.
\newblock Log canonical models of elliptic surfaces.
\newblock {\em Advances in Mathematics}, 320:210--243, 2017.

\bibitem[AB19]{tsm}
Kenneth Ascher and Dori Bejleri.
\newblock Moduli of fibered surface pairs from twisted stable maps.
\newblock {\em Mathematische Annalen}, 374(1-2):1007--1032, 2019.

\bibitem[AB21a]{master}
Kenneth Ascher and Dori Bejleri.
\newblock Moduli of weighted stable elliptic surfaces and invariance of log
  plurigenera.
\newblock {\em Proc. Lond. Math. Soc}, 122(5):617--677, 2021.

\bibitem[AB21b]{tsm2}
Kenneth Ascher and Dori Bejleri.
\newblock Smoothability of relative stable maps to stacky curves.
\newblock {\em arXiv:2108.05324}, 2021.

\bibitem[ABE20]{abe}
Valery {Alexeev}, Adrian {Brunyate}, and Philip {Engel}.
\newblock {Compactifications of moduli of elliptic K3 surfaces: stable pair and
  toroidal}.
\newblock {\em arXiv e-prints}, page arXiv:2002.07127, Feb 2020.

\bibitem[AET19]{aet}
Valery {Alexeev}, Philip {Engel}, and Alan {Thompson}.
\newblock {Stable pair compactification of moduli of K3 surfaces of degree 2}.
\newblock {\em arXiv e-prints}, page arXiv:1903.09742, Mar 2019.

\bibitem[Ale94]{boundedness}
Valery Alexeev.
\newblock Boundedness and {$K^2$} for log surfaces.
\newblock {\em Internat. J. Math.}, 5(6):779--810, 1994.

\bibitem[AV97]{av}
Dan Abramovich and Angelo Vistoli.
\newblock Complete moduli for fibered surfaces.
\newblock {\em Recent progress in intersection theory (Bologna, 1997)}, 1997.

\bibitem[AV02]{av2}
Dan Abramovich and Angelo Vistoli.
\newblock Compactifying the space of stable maps.
\newblock {\em J. Amer. Math. Soc.}, 15(1):27--75, 2002.

\bibitem[BB66]{bb}
W.~L. Baily, Jr. and A.~Borel.
\newblock Compactification of arithmetic quotients of bounded symmetric
  domains.
\newblock {\em Ann. of Math. (2)}, 84:442--528, 1966.

\bibitem[BHPVdV04]{complexsurfaces}
Wolf~P. Barth, Klaus Hulek, Chris A.~M. Peters, and Antonius Van~de Ven.
\newblock {\em Compact complex surfaces}, volume~4 of {\em Ergebnisse der
  Mathematik und ihrer Grenzgebiete. 3. Folge. A Series of Modern Surveys in
  Mathematics [Results in Mathematics and Related Areas. 3rd Series. A Series
  of Modern Surveys in Mathematics]}.
\newblock Springer-Verlag, Berlin, second edition, 2004.

\bibitem[Bru15]{brunyate}
Adrian Brunyate.
\newblock {\em A Modular Compactification of the Space of Elliptic K3
  Surfaces}.
\newblock PhD thesis, The University of Georgia, 2015.

\bibitem[Cad07]{cadman}
Charles Cadman.
\newblock Using stacks to impose tangency conditions on curves.
\newblock {\em Amer. J. Math.}, 129(2):405--427, 2007.

\bibitem[Dol96]{dolgachev}
I.~V. Dolgachev.
\newblock Mirror symmetry for lattice polarized {$K3$} surfaces.
\newblock {\em J. Math. Sci.}, 81(3):2599--2630, 1996.
\newblock Algebraic geometry, 4.

\bibitem[Fri84a]{friedman}
Robert Friedman.
\newblock A new proof of the global {T}orelli theorem for {$K3$} surfaces.
\newblock {\em Ann. of Math. (2)}, 120(2):237--269, 1984.

\bibitem[Fri84b]{friedmanbook}
Robert Friedman.
\newblock The period map at the boundary of moduli.
\newblock In {\em Topics in transcendental algebraic geometry ({P}rinceton,
  {N}.{J}., 1981/1982)}, volume 106 of {\em Ann. of Math. Stud.}, pages
  183--208. Princeton Univ. Press, Princeton, NJ, 1984.

\bibitem[Gat02]{gathmann}
Andreas Gathmann.
\newblock Absolute and relative {G}romov-{W}itten invariants of very ample
  hypersurfaces.
\newblock {\em Duke Math. J.}, 115(2):171--203, 2002.

\bibitem[GG14]{gg}
Noah Giansiracusa and William~Danny Gillam.
\newblock On {K}apranov's description of {$\overline M_{0,n}$} as a {C}how
  quotient.
\newblock {\em Turkish J. Math.}, 38(4):625--648, 2014.

\bibitem[Has03]{has}
Brendan Hassett.
\newblock Moduli spaces of weighted pointed stable curves.
\newblock {\em Advances in Mathematics}, pages 316--352, 2003.

\bibitem[HL02]{hl}
Gert Heckman and Eduard Looijenga.
\newblock The moduli space of rational elliptic surfaces.
\newblock In {\em Algebraic Geometry 2000}, volume~36 of {\em Adv. Stud. Pure
  Math}. Math. Soc. Japan, 2002.

\bibitem[HMX18]{hmx}
Christopher~D. Hacon, James McKernan, and Chenyang Xu.
\newblock Boundedness of moduli of varieties of general type.
\newblock {\em J. Eur. Math. Soc. (JEMS)}, 20(4):865--901, 2018.

\bibitem[HP15]{HP}
Robin Hartshorne and Claudia Polini.
\newblock Divisor class groups of singular surfaces.
\newblock {\em Trans. Amer. Math. Soc.}, 367(9):6357--6385, 2015.

\bibitem[HT15]{HT}
Andrew Harder and Alan Thompson.
\newblock The geometry and moduli of {K}3 surfaces.
\newblock In {\em Calabi-{Y}au varieties: arithmetic, geometry and physics},
  volume~34 of {\em Fields Inst. Monogr.}, pages 3--43. Fields Inst. Res. Math.
  Sci., Toronto, ON, 2015.

\bibitem[HX13]{hx1}
Christopher~D. Hacon and Chenyang Xu.
\newblock Existence of log canonical closures.
\newblock {\em Invent. Math.}, 192(1):161--195, 2013.

\bibitem[{Inc}20]{giovanni}
G.~{Inchiostro}.
\newblock {Moduli of Weierstrass fibrations with marked section}.
\newblock {\em Advances in Mathematics}, 2020.
\newblock To appear.

\bibitem[KK10]{kk}
J\'{a}nos Koll\'{a}r and S\'{a}ndor~J. Kov\'{a}cs.
\newblock Log canonical singularities are {D}u {B}ois.
\newblock {\em J. Amer. Math. Soc.}, 23(3):791--813, 2010.

\bibitem[KM98]{km}
J{\'a}nos Koll{\'a}r and Shigefumi Mori.
\newblock {\em Birational geometry of algebraic varieties}, volume 134 of {\em
  Cambridge Tracts in Mathematics}.
\newblock Cambridge University Press, Cambridge, 1998.
\newblock With the collaboration of C. H. Clemens and A. Corti, Translated from
  the 1998 Japanese original.

\bibitem[Kol13]{singmmp}
J{\'a}nos Koll{\'a}r.
\newblock {\em Singularities of the minimal model program}, volume 200 of {\em
  Cambridge Tracts in Mathematics}.
\newblock Cambridge University Press, Cambridge, 2013.
\newblock With a collaboration of S{\'a}ndor Kov{\'a}cs.

\bibitem[{Kol}18]{kollarmumford}
J{\'a}nos {Koll{\'a}r}.
\newblock {Mumford's influence on the moduli theory of algebraic varieties}.
\newblock {\em arXiv e-prints}, page arXiv:1809.10723, September 2018.

\bibitem[Kol21]{kolmodbook}
J\'anos Koll\'ar.
\newblock {\em Families of varieties of general type}.
\newblock Available at \url{https://web.math.princeton.edu/~kollar/}, 2021.

\bibitem[KP17]{kp}
S\'{a}ndor~J. Kov\'{a}cs and Zsolt Patakfalvi.
\newblock Projectivity of the moduli space of stable log-varieties and
  subadditivity of log-{K}odaira dimension.
\newblock {\em J. Amer. Math. Soc.}, 30(4):959--1021, 2017.

\bibitem[KSB88]{ksb}
J.~Koll{\'a}r and N.~I. Shepherd-Barron.
\newblock Threefolds and deformations of surface singularities.
\newblock {\em Invent. Math.}, 91(2):299--338, 1988.

\bibitem[Laz16]{laza}
Radu Laza.
\newblock The {KSBA} compactification for the moduli space of degree two {$K3$}
  pairs.
\newblock {\em J. Eur. Math. Soc. (JEMS)}, 18(2):225--279, 2016.

\bibitem[LN02]{ln}
Gabrielle La~Nave.
\newblock Explicit stable models of elliptic surfaces with sections.
\newblock {\em arXiv: 0205035}, 2002.

\bibitem[LZ16]{lazacpd}
Radu Laza and Zheng Zhang.
\newblock Classical period domains.
\newblock In {\em Recent advances in {H}odge theory}, volume 427 of {\em London
  Math. Soc. Lecture Note Ser.}, pages 3--44. Cambridge Univ. Press, Cambridge,
  2016.

\bibitem[Mir81]{mir}
Rick Miranda.
\newblock The moduli of weierstrass fibrations over $\mathbb{P}^1$.
\newblock {\em Math. Ann.}, 255(3):379--394, 1981.

\bibitem[Mir89]{mir3}
Rick Miranda.
\newblock {\em The basic theory of elliptic surfaces}.
\newblock Dottorato di Ricerca in Matematica. [Doctorate in Mathematical
  Research]. ETS Editrice, Pisa, 1989.

\bibitem[Nik79]{nikulin}
V.~V. Nikulin.
\newblock Finite groups of automorphisms of {K}\"ahlerian {$K3$} surfaces.
\newblock {\em Trudy Moskov. Mat. Obshch.}, 38:75--137, 1979.

\bibitem[Ols07]{olsson}
Martin~C Olsson.
\newblock (log) twisted curves.
\newblock {\em Compositio Mathematica}, 143(2):476--494, 2007.

\bibitem[OO18]{oo}
Y.~{Odaka} and Y.~{Oshima}.
\newblock {Collapsing K3 surfaces, Tropical geometry and Moduli
  compactifications of Satake, Morgan-Shalen type}.
\newblock {\em ArXiv e-prints}, October 2018.

\bibitem[Per90]{persson}
Ulf Persson.
\newblock Configurations of {K}odaira fibers on rational elliptic surfaces.
\newblock {\em Math. Z.}, 205(1):1--47, 1990.

\bibitem[Sca87]{scattone}
Francesco Scattone.
\newblock On the compactification of moduli spaces for algebraic {$K3$}
  surfaces.
\newblock {\em Mem. Amer. Math. Soc.}, 70(374):x+86, 1987.

\bibitem[Shi03]{shioda}
Tetsuji Shioda.
\newblock The elliptic k3 surfaces with a maximal singular fibre.
\newblock {\em Comptes Rendus Mathematique}, 337(7):461 -- 466, 2003.

\bibitem[SS09]{sselliptic}
M.~{Schuett} and T.~{Shioda}.
\newblock {Elliptic Surfaces}.
\newblock {\em ArXiv e-prints}, July 2009.

\bibitem[Vak00]{vakil}
Ravi Vakil.
\newblock The enumerative geometry of rational and elliptic curves in
  projective space.
\newblock {\em J. Reine Angew. Math.}, 529:101--153, 2000.

\end{thebibliography}
	
\end{document}